\documentclass[12pt,a4paper]{article}

\usepackage[dvips]{color}
\usepackage{color}
\usepackage{xcolor}

\usepackage{mathrsfs}
\usepackage{amsfonts}
\usepackage{mathtools}
\usepackage{stmaryrd}
\usepackage{amsmath}
\usepackage{amssymb}

\usepackage{bbm}
\usepackage{mathrsfs}
\usepackage{graphicx}

\usepackage{enumerate}
\usepackage{thm}

\makeatletter
\def\tank#1{\protected@xdef\@thanks{\@thanks
        \protect\footnotetext[0]{#1}}}
\def\bigfoot{

    \@footnotetext}
\makeatother

\newcommand{\ba}{\begin{eqnarray}}
\newcommand{\ea}{\end{eqnarray}}
\newtheorem{thm}{Theorem}[section]
\newtheorem{prop}{Proposition}[section]

\newtheorem{corollary}{Corollary}[section]
\newtheorem{lem}{Lemma}[section]
\newtheorem{defi}{Definition}[section]
\newtheorem{rmk}{Remark}[section]

\newtheorem{ass}{Assumption}[section]
\newtheorem{exmp}{Example}[section]
\numberwithin{equation}{section}

{\theorembodyfont{\rmfamily}

\newenvironment{proof}{Proof}{\hfill $\Box$}

\def\RR{\mathbb{R}}
\def\PP{\mathbb{P}}
\def\FF{\mathbb{F}}
\def\EE{\mathbb{E}}
\def\NN{\mathbb{N}}

\def\cA{{\mathcal A}}
\def\cB{{\mathcal B}}

\def\cE{{\mathcal E}}
\def\cF{{\mathcal F}}

\def\be{{\beta}}

\def\si{{\sigma}}

\def\et{{\eta}}

\def\Om{{\Omega}}

\def\al{{\alpha}}

\def\be{{\beta}}

\def\si{{\sigma}}

\def\eps{{\epsilon}}

\def\th{{\theta}}
\def\EE{\mathbb{ E}}
\def\th{{\theta}}
\def\si{{\sigma}}
\def\al{{\alpha}}

\def\Sgn{{\rm Sgn}}

\setcounter{equation}{0}

\allowdisplaybreaks

\renewcommand{\d}{d}

\begin{document}

\title{\Large \bf Irreducibility and ergodicity of SPDEs driven by pure jump noise}
\date{}

\author{{Jian Wang}$^1$\footnote{E-mail:wg1995@mail.ustc.edu.cn}~~~{Hao Yang}$^2$\footnote{E-mail:yanghao@hfut.edu.cn}~~~ {Jianliang Zhai}$^3$\footnote{E-mail:zhaijl@ustc.edu.cn}~~~ {Tusheng Zhang}$^{4}$\footnote{E-mail:Tusheng.Zhang@manchester.ac.uk}
\\
  \small  1. School of Mathematics,  Hangzhou Normal University, Hangzhou 311121, China,\\
 \small 2. School of Mathematics, Hefei University of Technology, Hefei, Anhui 230009, China. \\
 \small  3. School of Mathematical Sciences,
 \small University of Science and Technology of China,\\
 \small Hefei, Anhui 230026, China.\\
\small 4. Department of Mathematics, University of Manchester, \\
\small Oxford Road, Manchester, M13 9PL, UK.
}

\maketitle

\begin{center}
\begin{minipage}{140mm}
{\bf Abstract:}
The irreducibility is fundamental for the study of ergodicity of stochastic dynamical systems.
The existing methods on the irreducibility of stochastic partial differential equations (SPDEs) and stochastic differential equations (SDEs) driven by  pure jump noise  are basically along the same lines as that for the Gaussian case, which are not particularly suitable for jump noise. As a result, restrictive conditions are usually placed on the driving jump noise. Basically the driving noises are additive type and more or less in the class of stable processes.
\vskip 0.2cm

In this paper, we develop a new and effective method to obtain  the irreducibility of SPDEs and SDEs driven by multiplicative pure jump noise.
 The conditions placed on the coefficients and the driving noise are very mild, and in some sense they are necessary and sufficient.   As an application of our main results,
we remove all the restrictive conditions on the driving noises in the literature, and derive  new irreducibility results of a large class of equations driven
by pure jump noise, including SPDEs with locally monotone coefficients, SPDEs/SDEs with singular coefficients, nonlinear $\rm Schr\ddot{o}dinger$ equations etc. We emphasize that under our setting the driving noises could be compound Poisson processes, even allowed to be  infinite dimensional. As further applications of the main results, we obtain  the ergodicity  of multi-valued, singular stochastic evolution inclusions  such as   stochastic $1$-Laplacian
evolution(total variation flow), stochastic sign fast diffusion equation, stochastic minimal surface flow, stochastic curve shortening flow,  etc.

\vspace{3mm} {\bf Keywords:}
Irreducibility; pure jump noise; stochastic partial differential equations; ergodicity; locally monotone coefficients; singular coefficients;

\vspace{3mm} {\bf AMS Subject Classification (2020):}
60H15; 60G51; 37A25; 60H17.
\end{minipage}
\end{center}

\newpage
\tableofcontents

\renewcommand\baselinestretch{1.2}
\setlength{\baselineskip}{0.28in}
\section{Introduction and motivation}

Let $H$ be a topological space with Borel $\sigma$-field $\mathcal{B}(H)$, and let $\mathbb{X}:=\{X^x(t),t\geq0;x\in H\}$ be an $H$-valued  Markov process on some
probability space $(\Omega,\mathcal{F},\PP)$.
$\mathbb{X}$ is said to be  irreducible in $H$ if for each $t>0$ and $x\in H$
\begin{center}
$\PP(X^x(t)\in B)>0$ \quad for any non-empty open set $B$.
\end{center}
In this paper, we are concerned with  the irreducibility of stochastic partial differential equations (SPDEs) and stochastic differential equations (SDEs) driven by pure jump noise.

\vskip 0.2cm
The irreducibility is a fundamental property of stochastic dynamic systems.
The importance of the study of the irreducibility lies in  its relevance in the analysis of the ergodicity of  Markov processes.
The uniqueness of the invariant measures/ergodicity is usually obtained by proving irreducibility and the strong Feller property, or the asymptotic strong Feller property, or the $e$-property; see \cite{DZ 1996, D, DMT 1995, FM 1995,FGRT 2015,GT 2016, Hairer M 2006,Kapica,KPS 2010, PZ,T 2020}.
Irreducibility  also plays an indispensable role in establishing large deviations of the occupation measures of Markov processes, we refer the reader to  \cite{ JNPS 2015,JNPS 2018,JNPS 2021,MN 2018,NPX 2022, W}; and it also plays an important role in the study of the recurrence
of Markov processes; see \cite{DZ 1996}.

The study of the irreducibility of stochastic dynamical systems driven by Gaussian noise has a long history, and there is a large amount of literature devoted to this topic; see, for instance, the classical works \cite{D,PZ}, the books \cite{DZ 1996,DaPrato2004}, and \cite{EM2001,Flandoli 1997,FM 1995,PZ, Z}.
To obtain the irreducibility for stochastic equations driven by Gaussian noise, one usually needs to solve a control problem.
In doing so, three ingredients  play very important role:
the (approximate) controllability of the associated PDEs, Girsanov's transformation of  Wiener processes, and the support of  Wiener processes/stochastic convolutions  on  path spaces.

However, things become quite different when the driving noises are pure jump processes.
   Compared with the case of the Gaussian driving noise, there are few results on the irreducibility of the case of the  pure jump driving noise, because the systems
   behave drastically differently due to the appearance of jumps.
The existing methods on the irreducibility of the dynamical system driven by jump noise  are basically along the same lines as that of the Gaussian case. They heavily rely on the fact that the driving noises are additive type and  more or less in the class of stable processes. The use of such methods to deal with the case of other types of additive pure jump noises appears to be unclear, let alone the case of multiplicative noises. Actually the methods and techniques available for dynamical systems driven by Gaussian  noise are not well suited  for investigating the irreducibility of systems driven by jump type noise for two main reasons.
One is that there exist very few results on the support of the pure jump L\'{e}vy processes/the
stochastic convolutions on path spaces. Due to the discontinuity of trajectories, the characterization of the support of the pure jump processes is not as satisfactory as in the case of  Gaussian noise. The other is that Girsanov's transformation of the pure jump L\'{e}vy process is much less effective than that of Gaussian case, because the density of the Girsanov transform of a
 Poisson random measure is expressed in terms of  nonlinear invertible and predictable transformations, and is to
censore jumps or thin the size of jumps. So far, there is a lack of effective methods to obtain the irreducibility of
   stochastic equations driven by  pure jump noise. This strongly motivates the current paper.

\textbf{The existing results}

Now we mention the existing results on the  irreducibility of SPDEs  driven by pure jump noise.
To do this, we introduce the so-called cylindrical pure jump L\'evy processes defined by the orthogonal expansion
\begin{eqnarray}\label{eq Intro 1}
L(t)=\sum_{i}\beta_iL_i(t)e_i,\ \ t\geq0,
\end{eqnarray}
where $\{e_i\}$ is an orthonormal basis of a  separable Hilbert space $H$, $\{L_i\}$ are real valued i.i.d. pure jump L\'evy processes, and $\{\beta_i\}$ is a given sequence of non zero real numbers. 

The first paper dealing with the irreducibility of stochastic equations driven by  pure jump noise was
published in \cite{17}. The authors obtained the irreducibility of semilinear SPDEs with  Lipschitz coefficients. The driving noises they considered are the so-called cylindrical symmetric $\alpha$-stable processes, $\alpha\in(0,2)$, which have the form (\ref{eq Intro 1}) with  $\{L_i\}$ replaced by real valued i.i.d. symmetric $\alpha$-stable processes.
The authors in \cite{WXX} proved the irreducibility of stochastic real Ginzburg-Landau equation on torus $\mathbb{T}=\mathbb{R}\setminus \mathbb{Z}$ in $H:=\{h\in L^2(\mathbb{T}):\int_\mathbb{T}h(y)dy=0\}$ driven by cylindrical symmetric $\alpha$-stable processes with $\alpha\in(1,2)$; see \cite[Theorem 2.3]{WXX}. In the paper, the coefficients in (\ref{eq Intro 1}) are required to satisfy
(ii) on page 1182 of \cite{WXX}, i.e.,
\begin{eqnarray}\label{eq intro 3}
\alpha\in(1,2)
\end{eqnarray}
and
\begin{eqnarray}\label{eq intro 2}
C_1\gamma_i^{-\beta}\leq |\beta_i|\leq C_2\gamma_i^{-\beta}\text{ with }\beta>\frac{1}{2}+\frac{1}{2\alpha}\text{ for some positive constants }C_1\text{ and }C_2,
\end{eqnarray}
here $\{\gamma_i=4\pi^2|i|^2\}$ are the eigenvalues of the Laplace operator on $H$.
In \cite{WX} and \cite{DWX2020}, the authors established the irreducibility of stochastic reaction-diffusion equation and stochastic Burgers equation driven by the subordinated cylindrical Wiener process with a $\alpha/2$-stable subordinator, $\alpha\in(1,2)$, respectively. However, the restrictions as (\ref{eq intro 3}) and (\ref{eq intro 2}) on the driving noises  are
also required. 

In \cite{FHR 2016} the authors studied the irreducibility of some stochastic Hydrodynamical systems with bilinear term; see \cite[Theorem 3.5]{FHR 2016}. The driving noises $L$ they considered are also of the form (\ref{eq Intro 1}) and  satisfy
\begin{itemize}
  \item[(a)] The intensity measure $\mu$ of each component process $L_i$ satisfies that there exists a strictly monotone and $C^1$ function $q:(0,\infty)\rightarrow(0,\infty)$ such that
      \begin{eqnarray}
      &&\lim_{r\nearrow\infty}q(r)=0,\ \lim_{r\searrow0}q(r)=1,\text{ and }\mu(dz)=q(|z|)|z|^{-1-\theta}dz,\ \theta\in(0,2);\label{eq intro 4}\\
      &&\int_\mathbb{R}(1-q^{1/2}(|z|))^2\mu(dz)<\infty.\label{eq intro 5}
      \end{eqnarray}

  \item[(b)] There exist a certain $\epsilon\in(0,2)$ and $\vartheta\in[0,1/2)$ such that
        \begin{eqnarray}
      \sum_i(|\beta_i|+\beta_i^2\lambda_i^{-2\vartheta}+\beta_i^2\lambda_i^{\epsilon-1}+\beta_i^4\lambda_i^\epsilon)<\infty.\label{eq intro 6}
      \end{eqnarray}
Here $0<\lambda_1<\lambda_2<...$ are the eigenvalues associated with a positive self-adjoint operator appearing in the equations they studied.
\end{itemize}
Note that the driving noises $L$ could not cover cylindrical $\alpha$-stable processes.
As mentioned
in \cite{FHR 2016},  the case of stable driving noise is still an open problem. This problem is now solved as the application of the main result in this paper, see  Proposition \ref{Prop monotone} below.

Now we introduce the results on the  irreducibility of SDEs  driven by pure jump noise.
  In \cite{XZ}, the authors  studied the irreducibility of SDEs with singular coefficients driven by symmetric and rotationally invariant $\alpha$-stable processes with $\alpha\in(1,2)$.   This is the only paper to get the irreducibility of stochastic equations driven by multiplicative pure jump noise. In \cite{AP} the authors obtained the irreducibility of a class of multidimensional Ornstein-Uhlenbeck processes driven by additive pure jump noise $L$. $L$ is of the form $L(t)=L_1(t)+L_2(t)$,  where
 $L_1$ and $L_2$ are independent $d$-dimensional pure jump L\'evy processes, such that one of the following conditions is satisfied

 (1)  $L_1$ is a subordinate Brownian motion, and $L_2$ can be any pure jump L\'evy process or vanish;

 (2) $L_1$ is an anisotropic L\'evy process with independent symmetric one dimensional $\alpha$-stable components for $\alpha\in(0,2)$, and $L_2$ is a compound Poisson process.

 \textbf{Our contributions}

 Our main results are  Theorem \ref{thmmulti} and Theorem \ref{thmmulti-additive case}. In a few words, to get the irreducibility, we only need to impose the conditions under which the well posedness can be guaranteed and a nondegenerate condition on the intensity measure of the  driving L\'evy noise, i.e., Assumptions \ref{ass3} and \ref{assv} in Section \ref{Sec 2}  respectively for the case of the multiplicative noise and the case of  the additive noise.

 \textbf{(I)} The approach to prove the irreducibility we proposed
is completely different with the existing ones. Our approach gets rid of  solving the (approximate) controllability for the associated PDEs, does not need to establish the support of  L\'evy processes/stochastic convolutions  on  path spaces, and does not rely on Girsanov's transformation of  L\'evy processes. As a result, we removed all the restrictions placed on the driving noises in the literature; see Propositions \ref{Prop monotone} and \ref{prop 4.4} in this paper, and some details will be provided below.

\textbf{(II)}
We established new irreducibility results of a large class of equations driven
by  pure jump noise with mild requirements. For instance, Proposition \ref{Pro 4.3} below establishes the irreducibility of the nonlinear $\rm Schr\ddot{o}dinger$ equations, which covers both focusing and defocusing nonlinearity in the full subcritical
range of exponents. To the best of our knowledge, the corresponding results are not even known in the case of Gaussian driving noise.
The framework of Proposition \ref{Prop monotone} in Section 4
covers  SPDEs such as  stochastic porous medium equation, stochastic $p$-Laplace equation, stochastic fast diffusion equation,  stochastic 2D Navier-Stokes equation,   stochastic equations of non-Newtonian fluids,  and many other stochastic Hydrodynamical systems, most of which can not be covered by the existing results.

\textbf{(III)}
As a further application of the main result of our paper, combining with the $e$-property, we can obtain the uniqueness of invariant measures of the linear $\rm Schr\ddot{o}dinger$ equations and a class of multi-valued, singular stochastic evolution inclusions (see Proposition \ref{Pro 4.3} and Theorem \ref{thm ergo 1}). Examples include stochastic $1$-Laplacian
evolution(total variation flow), stochastic sign fast diffusion equation, stochastic minimal surface flow, stochastic curve shortening flow,  etc. It seems quite difficult to get these results with other means due to the lack of strong dissipativity of the equations.

\textbf{Comparison with the existing works}

Compared with the results in \cite{XZ}, Proposition \ref{prop 4.4} of this paper establishes the irreducibility for a class of SDE with singular coefficients driven by non-degenerate $\al$-stable-like L\'evy process with $\alpha\in(0,2)$. We stress that the study of the supercritical case $\alpha\in(0,1)$ is much harder and attracts a lot of attention.

  Except \cite{XZ}, Proposition \ref{Prop monotone} covers all of the other existing results. Furthermore the driving noises are required much weaker assumptions.  An example of the driving noises required in Proposition \ref{Prop monotone} is the form (\ref{eq Intro 1}) with
\begin{itemize}
  \item[(c)] {the intensity measure $\mu$ of each component process $L_i$ satisfies that there exist  $a\in S_\mu\cap (0,+\infty)$ and $b\in S_\mu\cap (-\infty,0)$ such that $a/b$ is an
irrational number;
      here $S_\mu$ is the support of $\mu$, that is, the set of $x\in\mathbb{R}$ such that $\mu(G)>0$ for any open set $G$ containing $x$;
      }

      \item[(d)] $\{\beta_i\}$ is a given sequence of non zero real numbers satisfying the conditions under which the well posedness can be proven.
\end{itemize}
We remark that  for (c) to hold, the measure $\mu$ is not necessary to be absolutely continuous with respect to the Lebesgue measure, and the driving noise could be compound Poisson processes.  The assumptions (c) and (d) remove technical assumptions appeared  in the existing results, such as (\ref{eq intro 3})-(\ref{eq intro 6}).




\textbf{We now describe the main idea of this paper.} Let $H$ be a separable Hilbert space, and let $\mathbb{X}:=\{X^x(t),t\geq0;x\in H\}$ be an $H$-valued
$\rm c\grave{a}dl\grave{a}g$ strong Markov process on some
probability space $(\Omega,\mathcal{F},\PP)$. For example, $X^x=(X^x(t),t\geq0)$ could be  the unique solutions to SPDEs/SDEs driven by pure jump L\'evy noise with initial data $x\in H$. For any $x,y\in H$, $T>0$ and $\kappa>0$, our aim is to  prove that
\begin{eqnarray}\label{In eq 01}
\mathbb{P}\Big(X^x(T)\in B(y,\kappa)\Big)>0.
\end{eqnarray}
Here, for any $h\in H$ and $l>0$, denote $B(h,l)=\{\hbar\in H: \|\hbar-h\|_H<l\}$.

To do this, we impose two main assumptions: Assumptions \ref{ass2} and \ref{ass3}.  Intuitively speaking, the first one is a weakly continuous assumption on $\mathbb{X}$ uniformly in the initial data.
The second one is a nondegenerate condition on the intensity measure of the  driving L\'evy noise,  which basically says that for any $\hbar,\overrightarrow{\hbar}\in H$, the neighbourhoods of $\overrightarrow{\hbar}$ can be reached with positive probability from $\hbar$  through a finite number of choosing jumps.

Applying Assumption \ref{ass2} to the given $y$ and $\kappa$, there exist $\epsilon_0:=\epsilon(y,\frac{\kappa}{2})\in(0,\frac{\kappa}{4})$ and $t_0:=t(y,\frac{\kappa}{2})>0$ such that for any $\hbar\in B(y,\epsilon_0)$,
\begin{eqnarray}\label{Eq idea 00}
\mathbb{P}\Big(\big\{X^\hbar(t)\in B(y,\frac{\kappa}{2}), \forall t\in[0,t_0]\big\}\Big)>0.
\end{eqnarray}
Therefore, set $T_0=T-\frac{t_0}{2}$, once we  prove that there exists $\widetilde{T}\in(T_0,T)$ such that
\begin{eqnarray}\label{In eq 03}
\mathbb{P}\big(X^x(\widetilde{T})\in B(y,\epsilon_0)\big)>0,
\end{eqnarray}
by the Markov property of $\mathbb{X}$, (\ref{In eq 01}) follow from (\ref{Eq idea 00}) and (\ref{In eq 03}), completing the proof.

We now explain the ideas of proving (\ref{In eq 03}). First, notice that there exists $\zeta\in H$ such that for any $\rho>0$
\begin{eqnarray}\label{In eq 02}
\PP\big(X^{x}(T_0) \in B(\zeta,\rho)\big) > 0.
\end{eqnarray}
By Assumption \ref{ass3}, $B(y,\epsilon_0)$ can be reached with positive probability from $\zeta$  through a finite number of choosing jumps.
Let $\sigma_i$ be the $i$-th jump time. One key step to obtain (\ref{In eq 03}) is to prove that there exist $\rho_0>0$, $\rho_1>0$, $q_1\in H$, and $T_1\in(T_0,T)$ such that
\begin{eqnarray}\label{In eq 05}
&&\PP\Big(\{X^{x}(T_0) \in B(\zeta,\rho_0)\}
\cap
\{X^{x}(t) \in B(\zeta,2\rho_0),\forall t\in[T_0,\sigma_1)\}\nonumber\\
&&\ \ \ \ \ \ \ \ \ \ \cap
\{X^{x}(\sigma_1)\in B(q_1,\frac{\rho_1}{2})\}
\cap
\{X^{x}(t) \in B(q_1,\rho_1),\forall t\in(\sigma_1,T_1]\}\Big)>0,
\end{eqnarray}
which implies that
\begin{eqnarray}\label{In eq 04}
\PP\Big(\{X^{x}(T_0) \in B(\zeta,\rho_0)\}
\cap
\{X^{x}(T_1) \in B(q_1,\rho_1)\}\Big)>0.
\end{eqnarray}
 To get (\ref{In eq 05}), the following claims will be used:

(C1) Assumption \ref{ass2}, (\ref{In eq 02}) and the Markov property of $\mathbb{X}$ imply that
$$\PP\Big(\{X^{x}(T_0) \in B(\zeta,\rho_0)\}
\cap
\{X^{x}(t) \in B(\zeta,2\rho_0),\forall t\in[T_0,\sigma_1)\}\Big)>0.$$

(C2) The first choosing jump  ensures that
 the neighbourhood of $q_1$, $B(q_1,\frac{\rho_1}{2})$, can be reached with positive probability from $B(\zeta,2\rho_0)$. \\
  To complete the proof of (\ref{In eq 05}), a further delicate argument is carried out, which requires an intricate cutoff procedure and employs  stopping time techniques, etc. The argument exploits the strong Markov property of $\mathbb{X}$, the fact that
  the jumps of the Poisson random measure on disjoint subsets are mutually independent, and the fact that with probability one, two independent L\'evy processes can not  jump simultaneously at any given moment, etc.  It also relies on carefully choosing moments and sizes of jumps of the driving noises.

 After getting (\ref{In eq 05}) and (\ref{In eq 04}), following a recursive procedure we are able to prove that there exist $\{q_i,i=1,2,...,n\}\subseteq H$, $\{\rho_i,i=1,2,...,n\}\subseteq (0,\infty)$ and $T_0<T_1<T_2<...<T_n<T$ such that
 \begin{eqnarray}\label{In eq 06}
\PP\Big(\{X^{x}(T_0) \in B(\zeta,\rho_0)\}
\cap_{i=1}^n
\{X^{x}(T_i) \in B(q_i,\rho_i)\}
\Big)>0.
\end{eqnarray}
Carefully choosing $B(q_n,\rho_n)\subset B(y,\epsilon_0)$, the above inequality implies that
$$
\PP\Big(\{X^{x}(T_n) \in B(y,\epsilon_0)\}\Big)
\geq
\PP\Big(\{X^{x}(T_n) \in B(q_n,\rho_n)\}\Big)>0.
$$
Therefore, (\ref{In eq 03}) holds, completing the proof.

 An important novelty of this article is that,  we find a nondegenerate condition on the intensity measure of the  driving L\'evy noises to prove the irreducibility.
A further novelty is that the main assumptions, Assumptions \ref{ass2} and \ref{ass3}, are imposed separately on the process $\mathbb{X}$ and the intensity measure of the  driving L\'evy noises. These two assumptions are basically independent of each other. Both of them are held for most of the applications.
Therefore, the approach we are proposing here is quite robust, and covers  SPDEs/SDEs with quite singular coefficients.

The paper is organized as follows. In Section 2, we will give the main framework and main results: Theorems \ref{thmmulti} and \ref{thmmulti-additive case}. Section 3 is devoted to the proof  of the main results. In Section 4, we provide applications to SDEs and SPDEs including many interesting physical models. Since  Assumptions \ref{ass2} and \ref{ass3} are basically independent of each other, Section 4 is divided into three parts: Subsection 4.1 presents examples of the additive driving noises  satisfying  Assumption \ref{assv} (the corresponding  Assumption \ref{ass3} in the setting  of the additive noise). Subsection 4.2 gives examples of the multiplicative driving noises  satisfying   Assumption \ref{ass3}. Subsections 4.3-4.5 are to provide examples of physical models  satisfying   Assumption \ref{ass2}. The irreducibility of many interesting physical models driven by pure jump L\'evy noise is established in Subsections 4.3-4.5. In Section 5, we provide interesting examples for which ergodicity can be established.

\section{Preliminaries and statements of the main results}\label{Sec 2}
In this section, we will introduce the framework and state the main results.  Let
\[
V \subset H \simeq H^* \subset V^*
\]
be a Gelfand triple, i.e., $\big(H, \langle \cdot, \cdot \rangle_H \big)$ is a separable Hilbert space and identified with its dual space $H^*$ by the Riesz isomorphism, $V$ is a reflexive Banach space that is continuously and densely embedded into $H$. If $_{V^*}\langle \cdot, \cdot \rangle_V$ denotes the dualization between $V$ and its dual space $V^*$, then it follows that
\[ _{V^*}\langle u, v \rangle_V = \langle u , v \rangle_H, \quad u \in H,\ v\in V.
\]
Let $(\Om, \cF, \FF,\PP)$, where $\FF = \{\cF_t\}_{t \geq 0}$, be a filtered probability space satisfying the usual conditions.

For a metric space $(X,d_X)$, the Borel $\si$-field on $X$ will be written as $\cB(X)$. For any $x\in X$ and $l\geq0$, denote $B(x,l)=\{y\in X: d_X(y,x)<l\}$ and $\overline{B(x,l)}=\{y\in X: d_X(y,x)\leq l\}$. If $I\subset\mathbb{R}$ is a time interval, we denote by $D(I,X)$ the space of all $\rm c\grave{a}dl\grave{a}g$ paths from $I$ to $X$.

Let $(Z, \cB(Z))$ be a metric space, and $\nu$ a given $\si$-finite  measure $\nu$ on it, that is, there exists $Z_n\in\mathcal{B}(Z), n\in\mathbb{N}$ such that $Z_n\uparrow Z$ and $\nu(Z_n)<\infty, \forall n\in\mathbb{N}$. Let $N: \cB(Z\times\RR^+) \times \Omega\rightarrow \bar{\NN} := \NN \cup \{0, \infty\}$ be a time homogeneous Poisson random measure on $(Z, \cB(Z) )$ with intensity measure $\nu$. For the existence of such Poisson random measure, we refer the reader to \cite{IW1989}.
We denote by $\tilde{N} (dz,dt) = N(dz,dt) - \nu(dz)dt$ the compensated Poisson random measure associated to $N$.

Now we consider the following SPDEs driven by pure jump noise:
\begin{eqnarray}\label{meq1}
  &&dX(t)= \cA (X(t)) dt +\int_{Z_1^c}\!\!\!\!\si(X(t-), z)\tilde{N}(dz,dt) + \int_{Z_1}\!\!\!\! \si(X(t-), z)N(dz,dt),\\
  &&X(0)=x,\nonumber
\end{eqnarray}
where  $\cA:V\rightarrow V^*$ and $\si: H\times Z \rightarrow H$ are Borel measurable mappings, and, for any $m\in\mathbb{N}$, $Z_m^c$ denotes the complement of $Z_m$ relative to $Z$.
\vskip 0.2cm
\begin{defi}
An $H$-valued $\rm c\grave{a}dl\grave{a}g$ $\FF$-adapted process $X$ is called a solution of (\ref{meq1}) if the following conditions are satisfied

(I) $X(t,\omega)\in V$ for $dt\otimes \PP$-almost all $(t,\omega)\in[0,\infty)\times\Omega$, where
$dt$ stands for the Lebesgue measures on $[0,\infty)$;

(II)  $\int_0^t|\cA(X(s))|_{V^*}ds+\int_0^t \int_{Z_1^c}|\si(X(s), z)|_H^2\nu(dz)ds
+
\int_0^t \int_{Z_1} |\si(X(s-), z)|_HN(dz, ds)<\infty, \ \forall t\geq0$, $\PP$-a.s.,

(III) $\PP$-a.s.
\ba\label{eq o 1}
X(t) &=& x + \int_0^t{\cA(X(s))ds}+\int_0^t \int_{Z_1^c}\si(X(s-), z)\tilde{N}(dz, ds) \notag \\
&&+ \int_0^t \int_{Z_1} \si(X(s-), z)N(dz, ds) , \quad t\geq 0,
\ea
as an equation in $V^*$.
\end{defi}
For notational convenience, we use the notation $X^x$ to indicate the solution of (\ref{meq1}) starting from $x$.
\vskip 0.2cm
Our starting point is the following assumption.
\begin{ass}\label{ass1}
For any $x \in H$, there exists a unique global solution $X^x=\big(X^x(t)\big)_{t\geq 0}$ to (\ref{meq1}) and  $\{X^x,x\in H\}$ forms a  strong Markov process.
\end{ass}
\vskip 0.2cm
\begin{rmk}
Our primary concern in this paper is the irreducibility and ergodicity of the solutions of SPDEs. We simply impose the assumption 2.1 which of course holds under many variants of standard assumptions on the coefficients. See the examples in Section 4 and 5.
\end{rmk}
\vskip 0.2cm
For any $m\in\mathbb{N}$, since $\nu(Z_m)<\infty$, (\ref{eq o 1}) can be rewritten as follows:
\begin{eqnarray*}
X^x(t) &=&
x + \int_0^t{\cA(X^x(s))ds}+\int_0^t \int_{Z_m^c}\si(X^x(s-), z)\tilde{N}(dz, ds)\notag \\
&&- \int_0^t \int_{Z_m\setminus Z_1}\si(X^x(s), z)\nu(dz)ds+ \int_0^t \int_{Z_m} \si(X^x(s-), z)N(dz, ds), \quad t\geq 0.
\end{eqnarray*}
Removing the big jumps in the above equation, consider
 \begin{eqnarray}\label{meq1-Zhai}
  &&dX_m(t)= \cA (X_m(t)) dt +\int_{Z_m^c}\!\!\!\!\si(X_m(t-), z)\tilde{N}(dz,dt) - \int_{Z_m\setminus Z_1}\!\!\!\! \si(X_m(t), z)\nu(dz)dt,\nonumber\\
  &&X_m(0)=x.
\end{eqnarray}
Set
$$
N(Z_m, t):=\int_0^t \int_{Z_m}N(dz, ds),\ t\geq0,
$$
and
\begin{eqnarray}\label{stop time 1}
\tau_m^i=\inf\{t\geq 0: N(Z_m, t)=i\},\ \ i\in\mathbb{N}.
\end{eqnarray}

\vskip 0.2cm
It is clear that $\{X^x(t),t\in[0,\tau_m^1)\}$ is a solution to (\ref{meq1-Zhai}) on $t\in[0,\tau_m^1)$.
The next result shows that the existence and uniqueness of the solution of equation (\ref{meq1-Zhai}) follows from that of equation (\ref{meq1}).
\vskip 0.2cm
\begin{prop}\label{Prop 2}
Fix an arbitrary $m\in\mathbb{N}$. For any $x\in H$, there exists a unique global solution $X^x_m=\big(X^x_m(t)\big)_{t\geq 0}$ to equation (\ref{meq1-Zhai}), which also has the strong Markov property.
\end{prop}
\vskip 0.2cm

\begin{proof}
 Note that  equation (\ref{meq1-Zhai}) does not involve the jumps of the Poisson random measure $N$ in the set $Z_m$. Therefore, the solution $X^x_m$ is continuous at the jumping times $\tau_m^i,i\in\mathbb{N}$.  $X^x_m$ can be constructed as follows:

For $t\in[0,\tau_m^1)$, set $X^x_m(t)=X^x(t)$. Define $X^x_m(\tau_m^1)=\lim_{t\uparrow \tau_m^1}X^x(t)$. Let $Y^1(t),t\geq\tau_m^1$ be the solution of the following SPDE:
\begin{eqnarray*}
Y^1(t) &=&
 X^x_m(\tau_m^1)+ \int_{\tau_m^1}^t{\cA(Y^1(s))ds}+\int_{\tau_m^1}^t \int_{Z_m^c}\si(Y^1(s-), z)\tilde{N}(dz, ds)\notag \\
&&- \int_{\tau_m^1}^t \int_{Z_m\setminus Z_1}\si(Y^1(s), z)\nu(dz)ds+ \int_{\tau_m^1}^t \int_{Z_m} \si(Y^1(s-), z)N(dz, ds), \quad t\geq {\tau_m^1}.
\end{eqnarray*}
For $t\in[\tau_m^1,\tau_m^2)$, set $X^x_m(t)=Y^1(t)$, and define $X^x_m(\tau_m^2)=\lim_{t\uparrow \tau_m^2}Y^1(t)$. Recursively, let $Y^i(t),t\geq\tau_m^i$, for $i\geq 2$, denote  the solution of the following SPDE:
\begin{eqnarray*}
Y^i(t) &=&
 X^x_m(\tau_m^i)+ \int_{\tau_m^i}^t{\cA(Y^i(s))ds}+\int_{\tau_m^i}^t \int_{Z_m^c}\si(Y^i(s-), z)\tilde{N}(dz, ds)\notag \\
&&- \int_{\tau_m^i}^t \int_{Z_m\setminus Z_1}\si(Y^i(s), z)\nu(dz)ds+ \int_{\tau_m^i}^t \int_{Z_m} \si(Y^i(s-), z)N(dz, ds), \quad t\geq {\tau_m^i}.
\end{eqnarray*}
Set $X^x_m(t)=Y^i(t)$ for $t\in[\tau_m^i,\tau_m^{i+1})$,  and define $X^x_m(\tau_m^{i+1})=\lim_{t\uparrow \tau_m^{i+1}}Y^i(t)$. Since $\nu(Z_m)<\infty$, we have $\lim_{i\rightarrow\infty}\tau_m^i=\infty$, $\mathbb{P}$-a.s..  The solution $X^x_m=\big(X^x_m(t)\big)_{t\geq 0}$ is uniquely determined.

The strong Markov property of $\{X^x_m,x\in H\}$ is implied by that of $\{X^x,x\in H\}$.
\end{proof}

\vskip 0.2cm
We denote by $\mathcal{G}_m$ the $\mathbb{P}$-completion of $\sigma\{N(U\cap Z_m,t), U\in\mathcal{B}(Z),t\geq0\}$, and
$\mathcal{G}^c_m$ the $\mathbb{P}$-completion of $\sigma\{N(U\cap Z^c_m,t), U\in\mathcal{B}(Z),t\geq0\}$.
Then $\mathcal{G}_m$ and $\mathcal{G}^c_m$ are independent. Since $X_m^x\in \mathcal{G}^c_m$, and $\tau_m^1\in\mathcal{G}_m$, we have the following result.
\vskip 0.2cm
\begin{prop}\label{Prop 3}
$\sigma\{X_m^x\}$ and $\sigma\{\tau_m^1\}$ are independent.
\end{prop}

\vskip 0.2cm

For any $x,y\in H$, $\eta>0$ and $m\in\mathbb{N}$, define $\mathbb{F}$-stopping times
\begin{eqnarray}\label{Stop time 01}
\tau_{x,y}^\eta=\inf\{t\geq 0: X^x(t)\not\in B(y,\eta)\}\ \ \text{ and }\ \ \tau_{x,y,m}^\eta=\inf\{t\geq 0: X^x_m(t)\not\in B(y,\eta)\}.
\end{eqnarray}
Since $X^x\in D([0,\infty),H)$ and $X^x_m\in D([0,\infty),H)$, $\mathbb{P}$-a.s., we have $\mathbb{P}(\tau_{x,x}^\eta>0)=1$ and $\mathbb{P}(\tau_ {x,x,m}^\eta>0)=1.$
\vskip 0.2cm
\begin{rmk}\label{Rem 1}
Note that $X^x(\tau_{x,y}^\eta)$ may not belong to $\overline{B(y,\eta)}$, in general the following statement is not true:
$$
\sup_{s\in[0,\tau_{x,y}^\eta]}\|X^x(s)-y\|_H\leq \eta\ \ \text{on}\ \{\tau_{x,y}^\eta>0\},\ \ \mathbb{P}\text{-a.s.}
$$
However, we have
$$
\sup_{s\in[0,\tau_{x,y}^\eta)}\|X^x(s)-y\|_H\leq \eta
\ \ \text{and}\ \ \sup_{s\in[0,\frac{\tau_{x,y}^\eta}{2}]}\|X^x(s)-y\|_H\leq \eta
\ \ \text{on}\ \{\tau_{x,y}^\eta>0\},\ \ \mathbb{P}\text{-a.s.}
$$

\end{rmk}
\vskip 0.2cm
We need the following assumption, which is held for most of the applications.
\begin{ass}\label{ass2}
For any $h\in H$, there exists $\eta_h>0$ such that, for any $\eta\in(0,\eta_h]$, there exist $(\epsilon,t)=(\epsilon(h,\eta),t(h,\eta))\in (0,\frac{\eta}{2}]\times (0,\infty)$ satisfying,
\begin{eqnarray*}
 \inf_{\tilde{h}\in B(h,\epsilon)}\mathbb{P}(\tau_{\tilde{h},h}^\eta\geq t)>0.
\end{eqnarray*}
\end{ass}

\begin{rmk}
The assumption \ref{ass2} basically says that with positive probability,  the solution stays in the neighbourhood of the initial position in a very short time period. This assumption is very mild, which holds for almost all reasonable equations.
\end{rmk}

\vskip 0.2cm

\begin{prop}\label{Prop 5}
Assume that Assumptions \ref{ass1} and \ref{ass2} hold. For any $m\in\mathbb{N}$ and $h\in H$, there exists $\eta^m_h>0$ such that, for any $\eta\in(0,\eta^m_h]$, there exist $(\epsilon,t)=(\epsilon(m,h,\eta),t(m,h,\eta))\in (0,\frac{\eta}{2}]\times (0,\infty)$ satisfying,
\begin{eqnarray}\label{2-1}
\inf_{\tilde{h}\in B(h,\epsilon)}\mathbb{P}(\tau_{\tilde{h},h,m}^\eta\geq t)>0.
\end{eqnarray}

\end{prop}

\begin{proof}
Recall $\tau_m^1$ defined by (\ref{stop time 1}). For any $\tilde{h},h\in H$, $\eta,s>0$, we have
\begin{eqnarray}\label{eq 722}
\mathbb{P}(\tau_{\tilde{h},h}^\eta\geq s)
=
\mathbb{P}(\{\tau_{\tilde{h},h}^\eta\geq s\}\cap\{\tau_m^1\leq s\})
+
\mathbb{P}(\{\tau_{\tilde{h},h}^\eta\geq s\}\cap\{\tau_m^1> s\}).
\end{eqnarray}
Note that $\{X^x(t),t\in[0,\tau_m^1)\}$ coincides with $\{X_m^x(t),t\in[0,\tau_m^1)\}$ on $t\in[0,\tau_m^1)$. We have
\begin{eqnarray*}
\{\tau_{\tilde{h},h}^\eta\geq s\}\cap\{\tau_m^1> s\}=\{\tau_{\tilde{h},h,m}^\eta\geq s\}\cap\{\tau_m^1> s\}.
\end{eqnarray*}
Applying Proposition \ref{Prop 3},
$$
\mathbb{P}(\{\tau_{\tilde{h},h,m}^\eta\geq s\}\cap\{\tau_m^1> s\})=\mathbb{P}(\tau_{\tilde{h},h,m}^\eta\geq s)\mathbb{P}(\tau_m^1> s).
$$
Combining this equality with (\ref{eq 722}), one concludes that
\begin{eqnarray}\label{eq prp 3 01}
\mathbb{P}(\tau_{\tilde{h},h,m}^\eta\geq s)
\geq
\frac{\mathbb{P}(\tau_{\tilde{h},h}^\eta\geq s)-\mathbb{P}(\tau_m^1\leq s)}{\mathbb{P}(\tau_m^1> s)}.
\end{eqnarray}
By Assumption \ref{ass2}, there exists $\eta_h>0$ such that, for any $\eta\in(0,\eta_h]$, there exist $(\epsilon,t)=(\epsilon(h,\eta),t(h,\eta))\in (0,\frac{\eta}{2}]\times (0,\infty)$ satisfying
\begin{eqnarray}\label{eq prp 3 02}
\beta:=\inf_{\tilde{h}\in B(h,\epsilon)} \mathbb{P}(\tau_{\tilde{h},h}^\eta\geq t)>0.
\end{eqnarray}
Then, for any $s\in(0,t]$,
\begin{eqnarray}\label{eq prp 3 03}
\inf_{\tilde{h}\in B(h,\epsilon)} \mathbb{P}(\tau_{\tilde{h},h}^\eta\geq s)\geq\beta.
\end{eqnarray}
Recall that  $\tau_m^1$ has the exponential distribution with parameter $\nu(Z_m)<\infty$, that is,
\begin{eqnarray}\label{eq prp 3 04}
\mathbb{P}(\tau_m^1> s)=e^{-\nu(Z_m)s},\ \ \ \ \mathbb{P}(\tau_m^1\leq s)=1-e^{-\nu(Z_m)s}.
\end{eqnarray}
Putting together  (\ref{eq prp 3 01})--(\ref{eq prp 3 04}) we see that there exists $s_0>0$ small enough such that
$$
\inf_{\tilde{h}\in B(h,\epsilon)} \mathbb{P}(\tau_{\tilde{h},h,m}^\eta\geq s_0)>0.
$$
This completes the proof.
\end{proof}

\vskip 0.2cm
Now we introduce the conditions on the jumping measure of the driving noise of the equation (\ref{meq1}), which basically says that for any $\hbar,y\in H$, the neighbourhoods of $y$ can be reached from $\hbar$  through a finite number of choosing jumps.

\begin{ass}\label{ass3}
For any $\hbar,y\in H$ with $\hbar\neq y$ and any $\bar{\eta}>0$, there exist $n,m\in\mathbb{N}$, and $\{l_i,i=1,2,...,n\}\subset Z_m$
 such that,
 for any $\eta\in(0,\frac{\bar{\eta}}{2})$, there exist $\{\epsilon_i,i=1,2,...,n\}\subset(0,\infty)$ and $\{\eta_i,i=0,1,...,n\}\subset(0,\infty)$ such that, denoting
$$
q_0=\hbar,\ q_{i}=q_{i-1}+\sigma(q_{i-1},l_{i}),\ i=1,2,...,n,
$$
 \begin{itemize}
\item  $0<\eta_0\leq \eta_1\leq ...\leq \eta_{n-1}\leq \eta_n\leq \eta$;
   \item for any $i=0,1,...,n-1$, $\{\tilde{q}+\sigma(\tilde{q},l):\ \tilde{q}\in B(q_i,\eta_i),l\in B(l_{i+1},\epsilon_{i+1})\}\subset B(q_{i+1},\eta_{i+1})$;
   \item $B(q_{n},\eta_{n})\subset B(y,\frac{\bar{\eta}}{2})$;
   \item for any $i=1,2,...,n$, $\nu(B(l_{i},\epsilon_i))>0$;
   \item there exists $m_0\geq m$ such that $\bigcup_{i=1}^n B(l_{i},\epsilon_{i})\subset Z_{m_0}$.
 \end{itemize}

\end{ass}

\begin{rmk}\label{Rem 2}
Note that
$$
\lim_{\eta\searrow 0}\eta_i=0,\ \ \forall i=0,1,...n.
$$
 Moreover, in the above assumption, without loss of generality, it is not difficult to see that we can require that $\epsilon_i, i=1,2,...n$ is non-increasing as $\eta\searrow 0$.
\end{rmk}

\vskip 0.2cm
Now we are in a  position to state the main result of the paper.
\begin{thm}\label{thmmulti}
Suppose  Assumptions  \ref{ass1}, \ref{ass2} and \ref{ass3} hold. Then the Markov process formed by the solution $\{X^{x},x\in H\}$ of equation (\ref{meq1}) is irreducible in $H$.
\end{thm}

\vskip 0.2cm
In the rest of this section, let us consider the particular case of the additive noise. Let  now $Z=H$, and $\nu$ a  given  $\si$-finite intensity measure of a L\'evy process  on $H$. Recall that $\nu(\{0\})=0$ and $\int_{H}(\|z\|_H^2\wedge1)\nu(dz)<\infty$.
Let $N: \cB(H\times\RR^+) \times \Omega\rightarrow \bar{\NN}$ be the time homogeneous Poisson random measure with intensity measure $\nu$. Again $\tilde{N} (dz,dt) = N(dz,dt) - \nu(dz)dt$ denotes the  compensated Poisson random measure associated to $N$.
\vskip 0.2cm
Let us point out that, see e.g., \cite[Theorems 4.23 and 6.8]{PZ 2007}, in this case  $$L(t)=\int_0^t\int_{0<\|z\|_H\leq1}z\tilde{N}(dz,ds) + \int_0^t\int_{\|z\|_H>1} zN(dz,ds), t\geq0$$ defines an $H$-valued L\'evy process.

\vskip 0.2cm

Let $Z_m=\{\chi\in H: \|\chi\|_H> \frac{1}{m}\}$, $m\in\mathbb{N}$. Now consider the SPDE (\ref{meq1}) with the additive noise $dL(t)$, that is,
\begin{eqnarray}\label{meq1-additive case}
  &&dX(t)= \cA (X(t)) dt +dL(t),\\
  &&X(0)=x.\nonumber
\end{eqnarray}

Let us now formulate the condition on the jumping measure in this setting.
\vskip 0.2cm

\begin{ass}\label{assv}
For any $h \in H$ and $\et_h >0$, there exist $n\in\mathbb{N}$, a sequence of strict positive numbers $\et_1,\ \et_2, \cdots, \et_n$, and $a_1,\ a_2, \cdots,\ a_n \in H\setminus\{0\}$, such that $ 0 \notin \overline{B(a_i, \et_i)}$, $\nu\big(B(a_i, \et_i)\big)>0$ , $ i=1,..., n$, and that $\sum_{i=1}^n{ B(a_i, \et_i)} :=\{\sum_{i=1}^n h_i:h_i\in B(a_i, \et_i),  1 \leq i \leq n\}\subset B(h, \et_h)$.
\end{ass}
\vskip 0.2cm

As an application of Theorem \ref{thmmulti}, we have
\begin{thm}\label{thmmulti-additive case}
Under Assumptions  \ref{ass1}, \ref{ass2} and \ref{assv}, the Markov process formed by the solution $\{X^{x},x\in H\}$ of equation (\ref{meq1-additive case}) is irreducible in $H$.
\end{thm}

\vskip 0.2cm

Considering (\ref{meq1-additive case}) with $\mathcal{A}=0$ and $x=0$, then we have
\begin{corollary}\label{cor 1}
 Assume that Assumptions \ref{assv} holds. For any $s>0$,  $\epsilon>0$ and $h\in H$, $\mathbb{P}(L(s)\in B(h,\epsilon))>0$.
\end{corollary}

\vskip 0.2cm


We also like to stress that Corollary \ref{cor 1} covers both finite and infinite dimensional L\'evy processes, and is  even new  for $\mathbb{R}^n$-valued L\'evy processes. We  refer the reader to Chapter 5 in \cite{Sato} for the study of the support of $\mathbb{R}^n$-valued L\'evy processes.

\section{Proofs of the main results}
In this section, we will give the proof of Theorem 2.1. To this end, we need to prepare a number of results. Recall $\tau_{x,y}^{\eta}$ introduced in (\ref{Stop time 01}).

\begin{prop}\label{Prop 6-1}
Assume that Assumptions \ref{ass1} and \ref{ass2} hold. For any $h\in H$, there exists $\eta_h>0$ such that, for any $\eta\in(0,\eta_h]$, there exist $s>0$, $\varpi_0\in(0,\frac{\eta}{2}]$
satisfying
for any $\tilde{h}\in \overline{B(h,\varpi_0)}$,
\begin{eqnarray}\label{eq prop 6 02-1}
\mathbb{P}(\sup_{l\in[0,s]}\|X^{\tilde{h}}(l)-h\|_H\leq \eta)>0.
\end{eqnarray}
\end{prop}
\begin{proof}
By Assumption 2.2, for any $h\in H$, there exists $\eta_h>0$ such that, for any $\eta\in(0,\eta_h]$, there exists $(\epsilon,t)=(\epsilon(h,\eta),t(h,\eta))\in (0,\frac{\eta}{2}]\times (0,\infty)$ such that, for any $\tilde{h}\in B(h,\epsilon)$,
\begin{eqnarray*}
\mathbb{P}(\tau_{\tilde{h},h}^\eta\geq t)>0,
\end{eqnarray*}
which implies that for any $\tilde{h}\in \overline{B(h,\frac{\epsilon}{2})}$,
\begin{eqnarray*}
\mathbb{P}(\sup_{l\in[0,\frac{t}{2}]}\|X^{\tilde{h}}(l)-h\|_H\leq \eta)>0.
\end{eqnarray*}
We simply choose $s=\frac{t}{2}$ and $\varpi_0=\frac{\epsilon}{2}$ to complete the proof.
\end{proof}

\vskip 0.2cm

By Proposition \ref{Prop 5}, a similar argument to that used for the proof of Proposition \ref{Prop 6-1} leads to the following statement.
\begin{prop}\label{Prop 6-01-1}
Assume that Assumptions \ref{ass1} and \ref{ass2} hold. For any $m\in\mathbb{N}$ and  $h\in H$, there exists $\eta^m_h>0$ such that, for any $\eta^m\in(0,\eta^m_h]$, there exist $s^m>0$, $\varpi^m_0\in(0,\frac{\eta^m}{2}]$ such that
for any $\tilde{h}\in \overline{B(h,\varpi^m_0)}$,
\begin{eqnarray}\label{eq prop 6 04-1}
\mathbb{P}(\sup_{l\in[0,s^m]}\|X^{\tilde{h}}_m(l)-h\|_H\leq \eta^m)>0.
\end{eqnarray}

\end{prop}

\begin{prop}\label{Prop 7}
Assume that Assumption \ref{ass3} holds. For any $\hbar,y\in H$ with $\hbar\neq y$ and any $\bar{\eta}>0$, there exist $n,m\in\mathbb{N}$, and $\{l_i,i=1,2,...,n\}\subset Z_m$
 such that, for any $\eta\in(0,\frac{\bar{\eta}}{2})$, there exist $\{\epsilon_i,i=1,2,...,n\}\subset(0,\infty)$ and
 $$
 0<\eta_0<\eta_1<2\eta_1\leq\eta_1'<\eta_2<2\eta_2\leq\eta_2'<...<\eta_{n-1}<2\eta_{n-1}\leq\eta_{n-1}'<\eta_n<2\eta_n\leq\eta_n'<\eta
 $$
  such that, denoting
$$
q_0=\hbar,\ q_{i}=q_{i-1}+\sigma(q_{i-1},l_{i}),\ i=1,2,...,n,
$$
 \begin{itemize}
   \item $\{\tilde{q}+\sigma(\tilde{q},l),\ \tilde{q}\in \overline{B(q_0,\eta_0)},l\in B(l_{1},\epsilon_{1})\}\subset \overline{B(q_{1},\eta_{1})}$,
   \item for any $i=1,...,n-1$, $\{\tilde{q}+\sigma(\tilde{q},l),\ \tilde{q}\in \overline{B(q_i,\eta_i')},l\in B(l_{i+1},\epsilon_{i+1})\}\subset \overline{B(q_{i+1},\eta_{i+1})}$;
   \item $\overline{B(q_{n},2\eta_{n})}\subset\overline{B(q_{n},\eta'_{n})}\subset B(y,\frac{\bar{\eta}}{2})$;
   \item for any $i=1,2,...,n$, $\nu(B(l_{i},\epsilon_i))>0$;
   \item there exists $m_0\geq m$ such that $\bigcup_{i=1}^n B(l_{i},\epsilon_{i})\subset Z_{m_0}$;
   \item for any $i=1,2,...,n$, $\epsilon_i$ is non-increasing as $\eta\searrow 0$.
 \end{itemize}

\end{prop}
\begin{proof}
By Assumption \ref{ass3} and Remark \ref{Rem 2}, for any $\hbar,y\in H$ with $\hbar\neq y$ and any $\bar{\eta}>0$, there exist $n,m\in\mathbb{N}$, and $\{l_i,i=1,2,...,n\}\subset Z_m$,
denoting
$$
q_0=\hbar,\ q_{i}=q_{i-1}+\sigma(q_{i-1},l_{i}),\ i=1,2,...,n,
$$
 such that, for any $\eta\in(0,\frac{\bar{\eta}}{2})$, setting $\eta_k=\frac{\eta}{2^k}$, $k\in \mathbb{N}$, there exist $\{\epsilon^k_i,i=1,2,...,n\}\subset(0,\infty)$ and $\{\eta^k_i,i=0,1,...,n\}\subset(0,\infty)$ satisfying
 \begin{itemize}
  \item  $0<\eta^k_0\leq \eta^k_1\leq ...\leq \eta^k_{n-1}\leq \eta^k_n\leq \eta_k$;
   \item for any $i=0,1,...,n-1$, $\{\tilde{q}+\sigma(\tilde{q},l),\ \tilde{q}\in B(q_i,\eta^k_i),l\in B(l_{i+1},\epsilon^k_{i+1})\}\subset B(q_{i+1},\eta^k_{i+1})$;
   \item $B(q_{n},\eta^k_{n})\subset B(y,\frac{\bar{\eta}}{2})$;
   \item for any $i=1,2,...,n$, $\nu(B(l_{i},\epsilon^k_i))>0$;
   \item for any $i=0,1,2,...,n$, $\lim_{k\rightarrow\infty}\eta^k_i=0$;
   \item for any $i=1,2,...,n$, $\epsilon^k_i$ is non-increasing as $k\rightarrow\infty$;
   \item there exists $m_0\geq m$ such that $\bigcup_{i=1}^n B(l_{i},\epsilon^k_{i})\subset Z_{m_0}$.
 \end{itemize}

Noting that, for any $q\in H, z\in Z$, $\varpi_1>\varpi_2\geq0$, $\beta_1\geq\beta_2\geq0$,
$$
\overline{B(q,\varpi_2)}\subseteq B(q,\varpi_1)
$$
and
$$
\{\tilde{q}+\sigma(\tilde{q},l),\ \tilde{q}\in \overline{B(q,\varpi_2)},l\in B(z,\beta_2)\}
\subseteq
 \{\tilde{q}+\sigma(\tilde{q},l),\ \tilde{q}\in B(q,\varpi_1),l\in B(z,\beta_1)\},
$$
 one can easily choose appropriate integers $k$, $\eta^k_i$ and $\epsilon^k_i$ above to get the positive numbers $\eta_i$, $\eta^{\prime}_i$ and  $\epsilon_i$ required in the statement of the proposition.
\end{proof}

\vskip 0.2cm

Combining Propositions \ref{Prop 6-1}, \ref{Prop 6-01-1} and \ref{Prop 7} together, we arrive at the following.
\begin{prop}\label{Prop 9-1}
Assume that Assumptions \ref{ass1}, \ref{ass2} and \ref{ass3} hold. For any $\hbar,y\in H$ with $\hbar\neq y$ and any $\bar{\eta}>0$, there exist  $n,m\in\mathbb{N}$, and $\{l_i,i=1,2,...,n\}\subset Z_{m}$ such that,
denoting
$$
q_0=\hbar,\ q_{i}=q_{i-1}+\sigma(q_{i-1},l_{i}),\ i=1,2,...,n,
$$
for any $\eta\in(0,\frac{\bar{\eta}}{2})$, there exist $s>0$, $\eta_i\in(0,\eta)$, $\varpi^1_i\in(0,\frac{\eta_i}{2}]$ and $\varpi^2_i\in(0,\frac{\varpi^1_i}{2}]$, $i=0,1,2,...,n$, and $\{\epsilon_i,i=1,2,...,n\}\subset(0,\infty)$, and $m_0\geq m$, such that, for any $s'\in(0,s]$,
 \begin{itemize}
 \item for any $i=0,1,2,...,n-1$, $\{\tilde{q}+\sigma(\tilde{q},l),\ \tilde{q}\in \overline{B(q_i,\eta_i)},l\in B(l_{i+1},\epsilon_{i+1})\}\subset \overline{B(q_{i+1},\varpi^2_{i+1})}$;
   \item $\overline{B(q_{n},\varpi^1_{n})}\subset\overline{B(q_{n},\eta_{n})}\subset B(y,\frac{\bar{\eta}}{2})$;
   \item for any $i=1,2,...,n$, $\nu(B(l_{i},\epsilon_i))>0$;
    \item for any $i=1,2,...,n$, $\epsilon_i$ is non-increasing as $\eta\searrow 0$;
   \item $\bigcup_{i=1}^n B(l_{i},\epsilon_{i})\subset Z_{m_0}$;

   \item for  $i=0,1,2,...,n$,
      \begin{eqnarray*}
       &&\mathbb{P}(\sup_{l\in[0,s']}\|X^{\tilde{h}}(l)-q_i\|_H\leq \varpi^1_i)>0,\ \forall\ {\tilde{h}}\in \overline{B(q_i,\varpi^2_i)};\\
       &&\mathbb{P}(\sup_{l\in[0,s']}\|X^{\tilde{h}}(l)-q_i\|_H\leq \eta_i)>0,\ \forall\ {\tilde{h}}\in \overline{B(q_i,\varpi^1_i)};
      \end{eqnarray*}
      and
       \begin{eqnarray*}
       &&\mathbb{P}(\sup_{l\in[0,s']}\|X^{\tilde{h}}_{m_0}(l)-q_i\|_H\leq \varpi^1_i)>0,\ \forall\ {\tilde{h}}\in \overline{B(q_i,\varpi^2_i)};\\
       &&\mathbb{P}(\sup_{l\in[0,s']}\|X_{m_0}^{\tilde{h}}(l)-q_i\|_H\leq \eta_i)>0,\ \forall\ {\tilde{h}}\in \overline{B(q_i,\varpi^1_i)}.
      \end{eqnarray*}

 \end{itemize}

\end{prop}

\noindent{\bf Proof of Theorem \ref{thmmulti}}
\vskip 0.2cm
\begin{proof}
We will prove that, for any $x,y\in H$, $T>0$ and $\kappa>0$,
\begin{eqnarray}\label{eq bukeyue}
\mathbb{P}(X^x(T)\in B(y,\kappa))>0.
\end{eqnarray}
 Fix now $x,y\in H$, $T>0$ and $\kappa>0$. Recall
$$
\tau_{\tilde{h},y}^\eta=\inf\{t\geq 0: X^{\tilde{h}}(t)\not\in B(y,\eta)\}.
$$
By Assumption 2.2, there exists $\eta_y>0$ such that, for any $\eta\in(0,\eta_y]$, there exists $(\epsilon_y^\eta,t_y^\eta)\in (0,\frac{\eta}{2}]\times (0,\infty)$ satisfying, for any $\tilde{h}\in B(y,\epsilon_y^\eta)$,
\begin{eqnarray*}
\mathbb{P}(\tau_{\tilde{h},y}^\eta\geq t_y^\eta)>0.
\end{eqnarray*}
Without loss of generality, we may assume
$$
\kappa<\eta_y.
$$
Now, let in particular $\eta=\frac{\kappa}{4}$ to get a pair $(\epsilon_y,t_y)\in (0,\frac{\kappa}{8}]\times (0,\infty)$ such that  for any $\tilde{h}\in B(y,\epsilon_y)$,
\begin{eqnarray*}
\mathbb{P}(\tau_{\tilde{h},y}^{\frac{\kappa}{4}}\geq t_y)>0.
\end{eqnarray*}
This implies that, for any $\tilde{h}\in B(y,\epsilon_y)$,
\begin{eqnarray}\label{eq bky 01}
\mathbb{P}(\sup_{s\in[0,\frac{t_y}{2}]}\|X^{\tilde{h}}(s)-y\|_H\leq \frac{\kappa}{4})
>0.
\end{eqnarray}
We may also assume $t_y\leq T$.
Set $T_0=T-\frac{t_y}{2}$. In the rest of the proof, we  distinguish the following two cases:

{\bf Case 1:} for any $\epsilon> 0$, $\PP\big(X^{x}(T_0) \in B(y,\epsilon)\big) > 0$;

{\bf Case 2:} there exists $\widehat{\epsilon} > 0$ such that $\PP\big(X^{x}(T_0) \in B(y,\widehat{\epsilon})\big) = 0$.
\vskip 0.2cm

For an $H$-valued measurable mapping
$S$ defined on the probability space $(\Omega,\mathcal{F},\PP)$, we will denote by $\PP^S$ the measure induced by $S$ on $(H,\mathcal{B}(H))$.
\vskip 0.2cm
We first consider the \textbf{Case 1}, namely,
assume that for any $\epsilon> 0$, $\PP\big(X^{x}(T_0) \in B(y,\epsilon)\big) > 0$. In particular, we have
\begin{eqnarray}\label{eq bky case 1 00}
\PP\big(X^{x}(T_0) \in B(y,\frac{\epsilon_y}{2})\big)>0.
\end{eqnarray}

By the Markov property of $\mathbb{X}:=\{X^x,x\in H\}$ and (\ref{eq bky case 1 00}), we have
\begin{eqnarray}\label{eq bky case 1 01}
&&\mathbb{P}(X^x(T)\in B(y,\kappa))\nonumber\\
&\geq&
\mathbb{P}(\{X^x(T_0)\in B(y,\frac{\epsilon_y}{2})\}\cap\{X^x(t)\in B(y,\kappa),\ \forall t\in[T_0,T]\})\nonumber\\
&=&
\int_{B(y,\frac{\epsilon_y}{2})}\PP(X^{\tilde{h}}(t)\in B(y,\kappa),\ \forall t\in[0,T-T_0])\PP^{X^{x}(T_0)}(d\tilde{h})\nonumber\\
&=&
\int_{B(y,\frac{\epsilon_y}{2})}\PP(X^{\tilde{h}}(t)\in B(y,\kappa),\ \forall t\in[0,\frac{t_y}{2}])\PP^{X^{x}(T_0)}(d\tilde{h}).
\end{eqnarray}

In view of (\ref{eq bky 01}), we have $\PP(X^{\tilde{h}}(t)\in B(y,\kappa),\ \forall t\in[0,\frac{t_y}{2}])>0$ for all $\tilde{h}\in B(y,\frac{\epsilon_y}{2})$. It follows from (\ref{eq bky case 1 00}) and (\ref{eq bky case 1 01}) that $\mathbb{P}(X^x(T)\in B(y,\kappa))>0$, completing the proof of the theorem in the \textbf{Case 1}.
\vskip 0.2cm

Next we consider the \textbf{Case 2}, namely,  assume that there exists $\widehat{\epsilon} > 0$ such that
\begin{eqnarray}\label{eq bky case 2 01}
\PP\big(X^{x}(T_0) \in B(y,\widehat{\epsilon})\big) = 0,
\end{eqnarray}
which implies that there exists $\zeta\not\in B(y,\widehat{\epsilon})$ such that for any $\rho>0$
\begin{eqnarray}\label{eq bky case 2 02-01}
\PP\big(X^{x}(T_0) \in B(\zeta,\rho)\big) > 0.
\end{eqnarray}

By Proposition \ref{Prop 9-1}, we have the following statements.

For the point $\zeta,y\in H$ and $\bar{\eta}=\kappa_0:=\frac{\epsilon_y}{2}$,
\begin{itemize}
\item[\textbf{(C)}]
there exist $n^{\zeta},m^{\zeta}\in\mathbb{N}$, and $\{l^{\zeta}_i,i=1,2,...,n^{\zeta}\}\subset Z_{m^{\zeta}}$,
denoting
$$
q^{\zeta}_0=\zeta,\ q^{\zeta}_{i}=q^{\zeta}_{i-1}+\sigma(q^{\zeta}_{i-1},l^{\zeta}_{i}),\ i=1,2,...,n^{\zeta},
$$
 such that, for $\eta=\frac{\kappa_0}{8}$, there exist $s^{\zeta}>0$, $\eta^{\zeta}_i\in(0,\eta)$, $\varpi^{1,{\zeta}}_i\in(0,\frac{\eta^{\zeta}_i}{2}]$ and $\varpi^{2,{\zeta}}_i\in(0,\frac{\varpi^{1,{\zeta}}_i}{2}]$, $i=0,1,2,...,n^{\zeta}$, and $\{\epsilon^{\zeta}_i,i=1,2,...,n^{\zeta}\}\subset(0,\infty)$, and $m^{\zeta}_0\geq m^{\zeta}$, such that, for $s'=s_0:=\frac{t_y}{4n^\zeta}\wedge \frac{s^\zeta}{2}$,
 \begin{itemize}
 \item[\textbf{(C-1)}] for any $i=0,1,2,...,n^{\zeta}-1$, $\{\tilde{q}+\sigma(\tilde{q},l),\ \tilde{q}\in \overline{B(q^{\zeta}_i,\eta^{\zeta}_i)},l\in B(l^{\zeta}_{i+1},\epsilon^{\zeta}_{i+1})\}\subset \overline{B(q^{\zeta}_{i+1},\varpi^{2,{\zeta}}_{i+1})}$;
   \item[\textbf{(C-2)}] $\overline{B(q^{\zeta}_{n^{\zeta}},\varpi^{1,{\zeta}}_{n^{\zeta}})}
                             \subset
                          \overline{B(q^{\zeta}_{n^{\zeta}},\eta^{\zeta}_{n^{\zeta}})}
                               \subset
                          B(y,\frac{\bar{\eta}}{2})=B(y,\frac{\epsilon_y}{4})$;
   \item[\textbf{(C-3)}] for any $i=1,2,...,n^{\zeta}$, $\nu(B(l^{\zeta}_{i},\epsilon^{\zeta}_i))>0$;
    \item[\textbf{(C-4)}]  for any $i=1,2,...,n^{\zeta}$, $\epsilon^{\zeta}_i$ is non-increasing as $\eta\searrow 0$;
   \item[\textbf{(C-5)}] $\bigcup_{i=1}^{n^{\zeta}} B(l^{\zeta}_{i},\epsilon^{\zeta}_{i})\subset Z_{m^{\zeta}_0}$;

   \item[\textbf{(C-6)}] for any $i=0,1,2,...,n^{\zeta}$,
      $$
       \mathbb{P}(\sup_{l\in[0,s_0]}\|X^{\tilde{h}}(l)-q^{\zeta}_i\|_H\leq \varpi^{1,{\zeta}}_i)>0,\ \ \forall {\tilde{h}}\in \overline{B(q^{\zeta}_i,\varpi^{2,{\zeta}}_i)};
       $$
       $$
       \mathbb{P}(\sup_{l\in[0,s_0]}\|X^{\tilde{h}}(l)-q^{\zeta}_i\|_H\leq \eta^{\zeta}_i)>0,
       \ \ \forall {\tilde{h}}\in \overline{B(q^{\zeta}_i,\varpi^{1,{\zeta}}_i)};
      $$
      and
       $$
       \mathbb{P}(\sup_{l\in[0,s_0]}\|X^{\tilde{h}}_{m^{\zeta}_0}(l)-q^{\zeta}_i\|_H\leq \varpi^{1,{\zeta}}_i)>0,
       \ \ \forall {\tilde{h}}\in \overline{B(q^{\zeta}_i,\varpi^{2,{\zeta}}_i)};
       $$
       $$
       \mathbb{P}(\sup_{l\in[0,s_0]}\|X_{m^{\zeta}_0}^{\tilde{h}}(l)-q^{\zeta}_i\|_H\leq \eta^{\zeta}_i)>0,
       \ \ \forall {\tilde{h}}\in \overline{B(q^{\zeta}_i,\varpi^{1,{\zeta}}_i)}.
      $$

 \end{itemize}
 \end{itemize}

 By (\ref{eq bky case 2 02-01}) and $\varpi^{2,{\zeta}}_0>0$, we have
\begin{eqnarray}\label{eq bky case 2 02}
\PP(X^{x}(T_0)\in B(\zeta,\varpi^{2,{\zeta}}_0))>0.
\end{eqnarray}

\vskip 0.2cm
 Set $T_i:=T_0+i s_0$. We will prove that, for any $i=1,2,...,n^\zeta$,
 \begin{eqnarray}\label{eq bky case 2 03}
 \PP(\{X^{x}(T_0)\in B(\zeta,\varpi^{2,{\zeta}}_0)\}\bigcap_{j=1}^i \{X^{x}(T_j)\in \overline{B(q_j^\zeta,\varpi^{1,{\zeta}}_j)}\})>0.
 \end{eqnarray}
 By the Markov property of $\mathbb{X}$,
  \begin{eqnarray}\label{eq bky case 2 04}
 &&\PP(\{X^{x}(T_0)\in B(\zeta,\varpi^{2,{\zeta}}_0)\}\bigcap\{X^{x}(T_1)\in \overline{B(q_1^\zeta,\varpi^{1,{\zeta}}_1)}\})\nonumber\\
 &=&\int_{B(\zeta,\varpi^{2,{\zeta}}_0)}\PP(X^\hbar(s_0)\in \overline{B(q_1^\zeta,\varpi^{1,{\zeta}}_1)})\PP^{X^{x}(T_0)}(d\hbar).
 \end{eqnarray}
  In view of (\ref{eq bky case 2 02}), to prove (\ref{eq bky case 2 03}) with $i=1$, it is sufficient to show that for any $\hbar\in B(\zeta,\varpi^{2,{\zeta}}_0)$,
  \begin{eqnarray}\label{eq bky case 2 05}
 \PP(X^\hbar(s_0)\in \overline{B(q_1^\zeta,\varpi^{1,{\zeta}}_1)})>0.
 \end{eqnarray}
 Let  $\sigma_1=\inf\{t\geq0: \int_0^t\int_{B(l^{\zeta}_{1},
 \epsilon^{\zeta}_{1})}N(dz,ds)=1\}$ be the first jumping time of the Poisson process $N(B(l^{\zeta}_{1},
 \epsilon^{\zeta}_{1}), [0,t])$. $\sigma_1$ has the exponential distribution with parameter $0<\nu(B(l^{\zeta}_{1},\epsilon^{\zeta}_{1}))<\nu(Z_{m^{\zeta}_0})<\infty$. In particular, we have
 \begin{eqnarray}\label{eq bky jump 1 st 01}
 \PP(\sigma_1\in(0,s_0])>0.
 \end{eqnarray}

 By the strong Markov property of $\mathbb{X}$, for any $\hbar\in B(\zeta,\varpi^{2,{\zeta}}_0)$,
   \begin{eqnarray}\label{eq bky case 2 06}
 &&\PP(X^\hbar(s_0)\in \overline{B(q_1^\zeta,\varpi^{1,{\zeta}}_1)})\nonumber\\
 &\geq&
  \PP(\{X^\hbar(s_0)\in \overline{B(q_1^\zeta,\varpi^{1,{\zeta}}_1)}\}
       \cap
       \{\sigma_1\in(0,s_0]\}
       \cap
       \{X^\hbar(\sigma_1)\in \overline{B(q_1^\zeta,\varpi^{2,{\zeta}}_1)}\})\\
 &\geq&
 \PP(
       \{\sigma_1\in(0,s_0]\}
       \cap
       \{X^\hbar(\sigma_1)\in \overline{B(q_1^\zeta,\varpi^{2,{\zeta}}_1)}\}
       \cap
       \{\sup_{s\in[\sigma_1,\sigma_1+s_0]}\|X^\hbar(s)-q_1^\zeta\|_H\leq\varpi^{1,{\zeta}}_1\}\nonumber
       )\\
 &=&
   \EE\Big(
       \EE\Big(
       1_{(0,s_0]}(\sigma_1)
       1_{\overline{B(q_1^\zeta,\varpi^{2,{\zeta}}_1)}}(X^\hbar(\sigma_1))
       1_{[0,\varpi^{1,{\zeta}}_1]}(\sup_{s\in[\sigma_1,\sigma_1+s_0]}\|X^\hbar(s)-q_1^\zeta\|_H)\nonumber
       \Big|\mathcal{F}_{\sigma_1}
       \Big)
   \Big)\\
    &=&
   \EE\Big(
1_{(0,s_0]}(\sigma_1)
       1_{\overline{B(q_1^\zeta,\varpi^{2,{\zeta}}_1)}}(X^\hbar(\sigma_1))
       \EE\Big(
       1_{[0,\varpi^{1,{\zeta}}_1]}(\sup_{s\in[0,s_0]}\|X^{\tilde{\hbar}}(s)-q_1^\zeta\|_H)\nonumber
       \Big|\tilde{\hbar}=X^\hbar(\sigma_1)
       \Big)
   \Big).
 \end{eqnarray}
 If we set $f(\tilde{\hbar})=\EE\Big(
       1_{[0,\varpi^{1,{\zeta}}_1]}(\sup_{s\in[0,s_0]}\|X^{\tilde{\hbar}}(s)-q_1^\zeta\|_H)\nonumber
       \Big)$
       , then
       \begin{eqnarray}\label{def condition expec}
       \EE\Big(
       1_{[0,\varpi^{1,{\zeta}}_1]}(\sup_{s\in[0,s_0]}\|X^{\tilde{\hbar}}(s)-q_1^\zeta\|_H)
       \Big|\tilde{\hbar}=X^\hbar(\sigma_1)
       \Big)
       =
       f(X^\hbar(\sigma_1)).
       \end{eqnarray}

  Using \textbf{(C-6)} for $i=1$, we have, for any $\tilde{\hbar}\in \overline{B(q_1^\zeta,\varpi^{2,{\zeta}}_1)}$,
 \begin{eqnarray*}
       \EE\Big(
       1_{[0,\varpi^{1,{\zeta}}_1]}(\sup_{s\in[0,s_0]}\|X^{\tilde{\hbar}}(s)-q_1^\zeta\|_H)
       \Big)>0.
 \end{eqnarray*}
 So $f(X^\hbar(\sigma_1))>0$ for $X^\hbar(\sigma_1)\in \overline{B(q_1^\zeta,\varpi^{2,{\zeta}}_1)}$.
 Hence, to prove (\ref{eq bky case 2 05}), in view of (\ref{eq bky case 2 06})  we only need to show, for any $\hbar\in B(\zeta,\varpi^{2,{\zeta}}_0)$,
 \begin{eqnarray}\label{eq bky case 2 07}
 \EE\Big(
     1_{(0,s_0]}(\sigma_1)
       1_{\overline{B(q_1^\zeta,\varpi^{2,{\zeta}}_1)}}(X^\hbar(\sigma_1))
     \Big)
     =
     \PP(\{\sigma_1\in(0,s_0]\}
       \cap
       \{X^\hbar(\sigma_1)\in \overline{B(q_1^\zeta,\varpi^{2,{\zeta}}_1)}\})
     >0.
 \end{eqnarray}
 Set $\tau_1=\inf\{t\geq 0:\int_0^t\int_{Z_{m^\zeta_0}\setminus B(l^{\zeta}_{1},\epsilon^{\zeta}_{1})}N(dz,ds)=1\}$.
 Since $\nu(Z_{m^\zeta_0})<\infty$,
  \begin{eqnarray}\label{eq bky jump 1 st 02}
 \PP(\tau_1>s_0)>0.
 \end{eqnarray}
For any $\hbar\in B(\zeta,\varpi^{2,{\zeta}}_0)$,
  \begin{eqnarray}\label{eq bky case 2 08}
     &&\PP(\{\sigma_1\in(0,s_0]\}
       \cap
       \{X^\hbar(\sigma_1)\in \overline{B(q_1^\zeta,\varpi^{2,{\zeta}}_1)}\})\nonumber\\
     &\geq&
     \PP(\{\sigma_1\in(0,s_0]\}
         \cap
          \{\tau_1>s_0\}
         \cap
       \{X^\hbar(\sigma_1)\in \overline{B(q_1^\zeta,\varpi^{2,{\zeta}}_1)}\})\\
        &\geq&
     \PP(\{\sigma_1\in(0,s_0]\}
         \cap
         \{\sup_{s\in[0,\sigma_1)}\|X^\hbar(s)-\zeta\|_H\leq \varpi^{1,{\zeta}}_0)\}
         \cap
          \{\tau_1>s_0\}
         \cap
       \{X^\hbar(\sigma_1)\in \overline{B(q_1^\zeta,\varpi^{2,{\zeta}}_1)}\}).\nonumber
 \end{eqnarray}
 Since $\nu(Z_{m^\zeta_0})<\infty$, $\nu(B(l^{\zeta}_{1},\epsilon^{\zeta}_{1}))<\infty$ and \textbf{(C-5)}, the solution to (\ref{eq o 1}) with $x=\hbar$ satisfies the following equation:
\begin{eqnarray*}
X^\hbar(t) &=&
\hbar + \int_0^t{\cA(X^\hbar(s))ds}+\int_0^t \int_{Z_{m^\zeta_0}^c}\si(X^\hbar(s-), z)\tilde{N}(dz, ds)\notag \\
&&- \int_0^t \int_{Z_{m^\zeta_0}\setminus Z_1}\si(X^\hbar(s), z)\nu(dz)ds
  + \int_0^t \int_{B(l^{\zeta}_{1},\epsilon^{\zeta}_{1})} \si(X^\hbar(s-), z)N(dz, ds)\\
&&+ \int_0^t \int_{Z_{m^\zeta_0}\setminus B(l^{\zeta}_{1},\epsilon^{\zeta}_{1})} \si(X^\hbar(s-), z)N(dz, ds), \quad t\geq 0.
\end{eqnarray*}
Notice that
\begin{eqnarray*}
\int_0^t \int_{B(l^{\zeta}_{1},\epsilon^{\zeta}_{1})} \si(X^\hbar(s-), z)N(dz, ds)=0,\text{ on }\{t\in[0,\sigma_1)\}, \ \ \ \  \PP\text{-a.s};
\end{eqnarray*}
and
 \begin{eqnarray*}
\int_0^t \int_{Z_{m^\zeta_0}\setminus B(l^{\zeta}_{1},\epsilon^{\zeta}_{1})} \si(X^\hbar(s-), z)N(dz, ds)=0,
\text{ on }\{t\in[0,\tau_1)\}, \ \ \  \PP\text{-a.s}.
\end{eqnarray*}
 Recalling the solution  $X_m$ introduced in (\ref{meq1-Zhai}) we conclude that for $s<\sigma_1\leq s_0<\tau_1$, $X^\hbar(s)=X^\hbar_{m^\zeta_0}(s)$. Moreover, by \textbf{(C-1)} with $i=0$, we see that
 \begin{eqnarray}
 &&\{\sup_{s\in[0,\sigma_1)}\|X^\hbar(s)-\zeta\|_H\leq \varpi^{1,{\zeta}}_0\}\nonumber\\
 &\subset& \{X^\hbar(\sigma_1-):=\lim_{s\uparrow\sigma_1}X^\hbar(s)\in  \overline{B(\zeta,\varpi^{1,{\zeta}}_0)}\}\nonumber\\
 &\subset& \{X^\hbar(\sigma_1)=X^\hbar(\sigma_1-)+(X^\hbar(\sigma_1)-X^\hbar(\sigma_1-))\in
 \overline{B(q^{\zeta}_{1},\varpi^{2,{\zeta}}_{1})}\}.
 \end{eqnarray}
 Here we have used the fact that
 \begin{eqnarray*}
 &&X^\hbar(\sigma_1)-X^\hbar(\sigma_1-)\\
 &=&
 \int_0^{\sigma_1} \int_{Z_{m^\zeta_0}^c}\si(X^\hbar(s-), z)\tilde{N}(dz, ds)
+
\int_0^{\sigma_1} \int_{B(l^{\zeta}_{1},\epsilon^{\zeta}_{1})} \si(X^\hbar(s-), z)N(dz, ds)\\
&&+
\int_0^{\sigma_1} \int_{Z_{m^\zeta_0}\setminus B(l^{\zeta}_{1},\epsilon^{\zeta}_{1})} \si(X^\hbar(s-), z)N(dz, ds)\\
&&-
\lim_{t\uparrow\sigma_1}\Big(
\int_0^t \int_{Z_{m^\zeta_0}^c}\si(X^\hbar(s-), z)\tilde{N}(dz, ds)
+
\int_0^t \int_{B(l^{\zeta}_{1},\epsilon^{\zeta}_{1})} \si(X^\hbar(s-), z)N(dz, ds)\\
&&\ \ \ \ \ \ \ \ \ \ \ \ +
\int_0^t \int_{Z_{m^\zeta_0}\setminus B(l^{\zeta}_{1},\epsilon^{\zeta}_{1})} \si(X^\hbar(s-), z)N(dz, ds)
\Big)
\\
&=&
\int_0^{\sigma_1} \int_{B(l^{\zeta}_{1},\epsilon^{\zeta}_{1})} \si(X^\hbar(s-), z)N(dz, ds)\\
 &\in&
 \{\sigma(X^\hbar(\sigma_1-),l),\ l\in B(l^{\zeta}_{1},\epsilon^{\zeta}_{1})\}.
 \end{eqnarray*}
 We therefore arrive at
  \begin{eqnarray}\label{eq bky jump 1 01}
&&\{\sigma_1\in(0,s_0]\}
         \cap
         \{\sup_{s\in[0,\sigma_1)}\|X^\hbar(s)-\zeta\|_H\leq \varpi^{1,{\zeta}}_0)\}
         \cap
          \{\tau_1>s_0\}
         \cap
       \{X^\hbar(\sigma_1)\in \overline{B(q_1^\zeta,\varpi^{2,{\zeta}}_1)}\}\nonumber\\
&=&
    \{\sigma_1\in(0,s_0]\}
         \cap
         \{\sup_{s\in[0,\sigma_1)}\|X^\hbar(s)-\zeta\|_H\leq \varpi^{1,{\zeta}}_0)\}
         \cap
          \{\tau_1>s_0\}\nonumber\\
&=&
    \{\sigma_1\in(0,s_0]\}
         \cap
         \{\sup_{s\in[0,\sigma_1)}\|X^\hbar_{m^\zeta_0}(s)-\zeta\|_H\leq \varpi^{1,{\zeta}}_0)\}
         \cap
          \{\tau_1>s_0\}\nonumber\\
&\supseteq&
    \{\sigma_1\in(0,s_0]\}
         \cap
         \{\sup_{s\in[0,s_0]}\|X^\hbar_{m^\zeta_0}(s)-\zeta\|_H\leq \varpi^{1,{\zeta}}_0)\}
         \cap
          \{\tau_1>s_0\}.
\end{eqnarray}
 Similar to the proof of Proposition \ref{Prop 3}, because the following events are determined by the jumps of the Poisson random measure on disjoints subsets,
 \begin{eqnarray}\label{eq bky jump 1 02}
 \{\sigma_1\!\in\!(0,s_0]\},\
         \{\!\!\sup_{s\in[0,s_0]}\!\!\!\|X^\hbar_{m^\zeta_0}(s)-\zeta\|_H\!\leq\! \varpi^{1,{\zeta}}_0)\} \text{ and }
          \{\tau_1\!>\!s_0\}
 \end{eqnarray}
  are mutually independent.
 Combining (\ref{eq bky jump 1 st 01}), (\ref{eq bky jump 1 st 02}), (\ref{eq bky case 2 08}), (\ref{eq bky jump 1 01}), and (\ref{eq bky jump 1 02}) together, we obtain that,  for any $\hbar\in B(\zeta,\varpi^{2,{\zeta}}_0)$,
  \begin{eqnarray}\label{eq bky jump 1 03}
   &&\PP(\{\sigma_1\in(0,s_0]\}
       \cap
       \{X^\hbar(\sigma_1)\in \overline{B(q_1^\zeta,\varpi^{2,{\zeta}}_1)}\})\nonumber\\
     &\geq&
     \PP(\sigma_1\in(0,s_0])
        \PP(\sup_{s\in[0,s_0]}\|X^\hbar_{m^\zeta_0}(s)-\zeta\|_H\leq \varpi^{1,{\zeta}}_0))
         \PP(\tau_1>s_0)\nonumber\\
        &>&0,
 \end{eqnarray}
 which proves (\ref{eq bky case 2 07}).
 \textbf{(C-6)} has been used in the last inequality.
 We have proved (\ref{eq bky case 2 03}) for $i=1$.

Now we prove (\ref{eq bky case 2 03}) with $i=2$. By the Markov property of $\mathbb{X}$,
 \begin{eqnarray}\label{eq bky case 2 i2 01}
 &&\PP(\{X^{x}(T_0)\in B(\zeta,\varpi^{2,{\zeta}}_0)\}\bigcap_{j=1}^2 \{X^{x}(T_j)\in \overline{B(q_j^\zeta,\varpi^{1,{\zeta}}_j)}\})\nonumber\\
 &=&
    \EE\Big(
        \EE\Big(
            1_{B(\zeta,\varpi^{2,{\zeta}}_0)}(X^{x}(T_0))
            1_{\overline{B(q_1^\zeta,\varpi^{1,{\zeta}}_1)}}(X^{x}(T_1))
            1_{\overline{B(q_2^\zeta,\varpi^{1,{\zeta}}_2)}}(X^{x}(T_2))
            \Big|\mathcal{F}_{T_1}
            \Big)
       \Big)\nonumber\\
  &=&
    \EE\Big(
    1_{B(\zeta,\varpi^{2,{\zeta}}_0)}(X^{x}(T_0))
            1_{\overline{B(q_1^\zeta,\varpi^{1,{\zeta}}_1)}}(X^{x}(T_1))
        \EE\Big(
            1_{\overline{B(q_2^\zeta,\varpi^{1,{\zeta}}_2)}}(X^{x}(T_2))
            \Big|\mathcal{F}_{T_1}
            \Big)
       \Big)\\
   &=&
    \EE\Big(
            1_{B(\zeta,\varpi^{2,{\zeta}}_0)}(X^{x}(T_0))
            1_{\overline{B(q_1^\zeta,\varpi^{1,{\zeta}}_1)}}(X^{x}(T_1))
        \EE\Big(
            1_{\overline{B(q_2^\zeta,\varpi^{1,{\zeta}}_2)}}(X^{\hbar}(s_0))
            \Big|\hbar=X^{x}(T_1)
            \Big)
       \Big).\nonumber
 \end{eqnarray}
 Here,
 \begin{eqnarray}
 \EE\Big(
            1_{\overline{B(q_2^\zeta,\varpi^{1,{\zeta}}_2)}}(X^{\hbar}(s_0))
            \Big|\hbar=X^{x}(T_1)
            \Big)
            :=\EE\Big(
            1_{\overline{B(q_2^\zeta,\varpi^{1,{\zeta}}_2)}}(X^{\hbar}(s_0))
            \Big)\Big|_{\hbar=X^{x}(T_1)}.
 \end{eqnarray}

(\ref{eq bky case 2 03}) with $i=1$ says that
\small{$$
\EE\Big(
    1_{B(\zeta,\varpi^{2,{\zeta}}_0)}(X^{x}(T_0))
            1_{\overline{B(q_1^\zeta,\varpi^{1,{\zeta}}_1)}}(X^{x}(T_1))
    \Big)
\!=\!
 \PP(\{X^{x}(T_0)\!\in\! B(\zeta,\varpi^{2,{\zeta}}_0)\}\cap \{X^{x}(T_1)\!\in\! \overline{B(q_1^\zeta,\varpi^{1,{\zeta}}_1)}\})
\!>\!0.
$$}
In view of (\ref{eq bky case 2 i2 01}), to prove (\ref{eq bky case 2 03}) for $i=2$, we only need to show that, for any $\hbar\in \overline{B(q_1^\zeta,\varpi^{1,{\zeta}}_1)}$,
\begin{eqnarray*}
\EE\Big(   1_{\overline{B(q_2^\zeta,\varpi^{1,{\zeta}}_2)}}(X^{\hbar}(s_0))
            \Big)
=
\PP\Big(
        X^{\hbar}(s_0)\in \overline{B(q_2^\zeta,\varpi^{1,{\zeta}}_2)}
    \Big)>0.
\end{eqnarray*}
This can be proved similarly as the proof of
 (\ref{eq bky case 2 05}). Thus we have proved
(\ref{eq bky case 2 03}) also for $i=2$. Following a recursive procedure we are able to prove (\ref{eq bky case 2 03}) for
any $i=1,2,...,n^\zeta$.

Now we will prove that
\begin{eqnarray}\label{eq bky case 2 f1}
&&\hspace{-1truecm}\PP\Big(\!
    \{X^{x}(T_0)\!\in\! B(\zeta,\varpi^{2,{\zeta}}_0)\}
    \!\!\bigcap_{j=1}^{n^\zeta-1}\!\!
    \{X^{x}(T_j)\!\in\! \overline{B(q_j^\zeta,\varpi^{1,{\zeta}}_j)}\}\!
    \bigcap
      \{X^{x}(T_{n^\zeta})\!\in\! B(y,\frac{\epsilon_y}{2})\}
    \bigcap
    \{X^{x}(T)\in B(y,\kappa)\}
\!\Big)\nonumber\\
&>&0.
\end{eqnarray}

Recall $s_0=\frac{t_y}{4n^\zeta}\wedge \frac{s^\zeta}{2}$, $\kappa_0=\frac{\epsilon_y}{2}$  and $T_0=T-\frac{t_y}{2}$. By \textbf{(C-2)} and (\ref{eq bky case 2 03}),
we have
\begin{eqnarray}\label{eq bky case 2 f2}
&&\PP\Big(
    \{X^{x}(T_0)\in B(\zeta,\varpi^{2,{\zeta}}_0)\}
    \bigcap_{j=1}^{n^\zeta-1}
    \{X^{x}(T_j)\in \overline{B(q_j^\zeta,\varpi^{1,{\zeta}}_j)}\}
    \bigcap
      \{X^{x}(T_{n^\zeta})\in B(y,\frac{\epsilon_y}{2})\}
\Big)\nonumber\\
&\geq&
\PP\Big(
    \{X^{x}(T_0)\in B(\zeta,\varpi^{2,{\zeta}}_0)\}
    \bigcap_{j=1}^{n^\zeta}
    \{X^{x}(T_j)\in \overline{B(q_j^\zeta,\varpi^{1,{\zeta}}_j)}\}
\Big)\nonumber\\
&>&
0.
\end{eqnarray}
Noticing that $T_{n^\zeta}=T_0+n^\zeta s_0\in(T_0,T)$ and applying the Markov property of $\mathbb{X}$ again,
we have
\small{
\begin{eqnarray}\label{eq bky case 2 f3}
&&\hspace{-1truecm}\PP\Big(\!
    \{X^{x}(T_0)\!\in\! B(\zeta,\varpi^{2,{\zeta}}_0)\}
    \!\!\bigcap_{j=1}^{n^\zeta-1}\!\!
    \{X^{x}(T_j)\!\in\! \overline{B(q_j^\zeta,\varpi^{1,{\zeta}}_j)}\}
    \bigcap
      \{X^{x}(T_{n^\zeta})\!\in\! B(y,\frac{\epsilon_y}{2})\}
    \bigcap
    \{X^{x}(T)\in B(y,\kappa)\}
\!\Big)\nonumber\\
&=&
  \EE\Big(
         1_{B(\zeta,\varpi^{2,{\zeta}}_0)}(X^{x}(T_0))
         \prod_{j=1}^{n^\zeta-1}1_{\overline{B(q_j^\zeta,\varpi^{1,{\zeta}}_j)}}(X^{x}(T_j))
         1_{B(y,\frac{\epsilon_y}{2})}(X^{x}(T_{n^\zeta}))\nonumber\\
         &&\hspace{7truecm}\cdot
           \EE\Big(
            1_{B(y,\kappa)}(X^{\hbar}(T-T_{n^\zeta}))
            \Big|\hbar=X^{x}(T_{n^\zeta})
            \Big)
      \Big).
\end{eqnarray}}
In view of (\ref{eq bky case 2 f2}) and (\ref{eq bky case 2 f3}), to prove (\ref{eq bky case 2 f1}), we only need to prove that,
for any $\hbar\in B(y,\frac{\epsilon_y}{2})$,
\begin{eqnarray}\label{eq bky case 2 f4}
           \EE\Big(
            1_{B(y,\kappa)}(X^{\hbar}(T-T_{n^\zeta}))
      \Big)
=
\PP\Big(
      X^{\hbar}(T-T_{n^\zeta})\in B(y,\kappa)
    \Big)
>0.
\end{eqnarray}
As $T-T_{n^\zeta}\in(0,\frac{t_y}{2})$, (\ref{eq bky case 2 f4}) follows from the choice of $t_y$, see (\ref{eq bky 01}). Hence (\ref{eq bky case 2 f1}) is established, which in particular yields
$$
\PP(
    X^{x}(T)\in B(y,\kappa)
)
>0.
$$

The proof is finished in the \textbf{Case 2}, completing the whole proof of Theorem \ref{thmmulti}.
\end{proof}

\section{Applications I: irreducibility}

In this section, we provide applications of our main results to various SDEs and SPDEs. Since  Assumptions \ref{ass2} and \ref{ass3} are basically independent of each other, this section is divided into three parts: Subsection 4.1 presents examples of the additive driving noises  satisfying   Assumption \ref{assv} (the corresponding Assumption \ref{ass3} in the case of the additive noise). Subsection 4.2 gives examples of multiplicative driving noises  satisfying   Assumption \ref{ass3}. Subsections 4.3-4.6 are to provide examples of physical models  satisfying   Assumption \ref{ass2}.  Combination of Subsections 4.1-4.6. produces plenty of examples for the irreducibility of SDEs and SPDEs driven by pure jump L\'evy noise, including  many interesting physical models.

\subsection{Sufficient conditions and examples for Assumption \ref{assv}}\label{Sub 4-1}
For any measure $\rho$ on $H$, its support $S_\rho=S(\rho)$ is defined to be the set of $x\in H$ such that $\rho(G)>0$ for any open set $G$ containing $x$.
Set
\begin{equation}\label{a-1}
H_0:=\Big\{\sum_{i=1}^n m_ia_i,\ n,m_1,...,m_n\in\mathbb{N},\ a_i\in S_\nu\Big\}.
\end{equation}
It is not difficult to see that Assumption \ref{assv} holds if and only if
$H_0$ is dense in $H$.
\vskip 0.2cm

 Assumption \ref{assv}  actually places very mild conditions on the intensity measures $\nu$ of the L\'evy processes.
The examples we can include are considerably  more general than that considered in the literature where  the irreducibility of SPDEs/SDEs driven by pure jump noise were studied. In this subsection, we give several explanatory examples that satisfy Assumption \ref{assv}.

\vskip 0.2cm\begin{exmp}\label{Prop Ex R}
Let $H=\mathbb{R}$.
The intensity measure $\nu$ of the L\'evy process satisfies Assumption \ref{assv}, namely, $H_0$ defined in (\ref{a-1}) is dense in $\mathbb{R}$, if one of the following conditions is satisfied:
\begin{itemize}
  \item[(1)] There exist $a<0$, $b>0$ and $c_n\neq0, n\in\mathbb{N},$ such that $\lim_{n\rightarrow\infty}c_n=0$, $\{a, b, c_n, n\in \mathbb{N}\}\subseteq S_\nu$.

  \item[(2)] $\nu(\mathbb{R})=\infty$, and there exist $a>0$ and $b<0$ such that $\{a,b\}\subseteq S_\nu$.

  \item[(3)] Set $S_\nu^+=\{a\in S_\nu:a>0\}$ and $S_\nu^-=\{a\in S_\nu:a<0\}$. There exist $a\in S_\nu^+$ and $b\in S_\nu^-$ such that $a/b$ is an irrational number.

%

%
\end{itemize}

%
\vskip 0.2cm
The proofs of statements  are elementary and omitted here. We stress that
it is easy to find many examples  with $\nu(\mathbb{R})<\infty$ that satisfy  Assumption \ref{assv}. This  means that the driving  L\'evy process $L$ could be a compound Poisson process on $\mathbb{R}$.

\end{exmp}

%
%
%
%

\vskip 0.2cm\begin{exmp}\label{Prop 3.3}
      Let $H=\mathbb{R}^d$, $d\in \mathbb{N}\cup{\{+\infty\}}$. Let $\{e_1,e_2,...,e_d\}$ be an orthonormal basis of $H$.  Let $\{L_i(t),t\geq0\}_{i\in\mathbb{N}}$ be mutually independent one dimensional pure jump L\'evy processes
      with their intensity measures $\nu_i$ satisfying one of the conditions listed in Example \ref{Prop Ex R}.
        Choosing $\beta_i\in \mathbb{R}\setminus\{0\},i\in\mathbb{N}$ such that
       \begin{eqnarray}\label{Eq ITM 01}
       \sum_{i=1}^d\int_{\mathbb{R}}|\beta_ix_i|^2\wedge1\nu_i(dx_i)<\infty.
       \end{eqnarray}
        $L(t)=\sum_{i=1}^d \beta_iL_i(t)e_i, t\geq0$ defines an $H$-valued L\'evy process, and its intensity measure $\nu$ satisfies Assumption \ref{assv}.
   \end{exmp}
%

\vskip 0.2cm
The following example is concerned with the subordination of L\'evy processes, which is an important way
to obtain new L\'evy processes.

\vskip 0.2cm\begin{exmp}\label{Prop exm sub}
Let $H=\mathbb{R}^d, d\in\mathbb{N}\cup\{+\infty\}$.
Let $Z=\{Z_t,t\geq0\}$ be a subordinator (an increasing L\'evy process on $\mathbb{R}$) with L\'evy measure $\rho$, drift $\beta_0$, which satisfy
$$
\beta_0\geq0\text{ and }\int_{(0,\infty)}(1\wedge s)\rho(ds)<\infty.
$$
Let $X=\{X_t,t\geq0\}$ be a L\'evy process on $H$ with intensity measure $\nu_X$. $Z$ and
$X$ are independent. Define
\begin{eqnarray}\label{eq zhai 01}
L_t(\omega)=X_{Z_t(\omega)}(\omega),\ t\geq0,\omega\in\Omega.
\end{eqnarray}

Then $\{L_t,t\geq0\}$ is a L\'evy process on $H$, and its intensity measure $\nu$ satisfies that
$$
\nu(B)=\beta_0\nu_X(B)+\int_{(0,\infty)}\mu^s_X(B)\rho(ds),\ \ B\in\mathcal{B}(H).
$$
Here $\mu^s_X$ is the law of $X_s$. See \cite[Theorem 30.1]{Sato} for details.

 If one of the following conditions holds, then the intensity measure $\nu$ of (\ref{eq zhai 01}) satisfies Assumption \ref{assv} .
\begin{itemize}
  \item one of $\int_{(0,\infty)}(1\wedge s)\rho(ds)$ and $\beta_0$ is not equal to $0$, and Assumption \ref{assv} holds  with
  $\nu$ replaced by $\nu_X$.
  \item $\int_{(0,\infty)}(1\wedge s)\rho(ds)>0$ and $S_{\mu^1_X}$ is dense in $H$.
\end{itemize}

\vskip 0.2cm

Now let $\{Z_t,t\geq0\}$ be a subordinator with L\'evy measure $\rho$ satisfying $\int_{(0,\infty)}(1\wedge s)\rho(ds)>0$. We have the following two concrete examples.

(1) Let $\{X_t,t\geq0\}=\{W_t,t\geq0\}$ be a $Q$-Wiener process on $H$, $Q\in L(H)$ is nonnegative, symmetric, with finite trace and $Ker Q=\{0\}$; here $L(H)$ denotes the set of all bounded linear operators on $H$.

(2) Let $\{X_t,t\geq0\}$ be a L\'evy process introduced in Example \ref{Prop 3.3}.

\end{exmp}

\vskip 0.2cm\begin{exmp}\label{Prop 04-19}

Let $H=\mathbb{R}^d, d\in\mathbb{N}$.  Assume that the  intensity measure $\nu$ of the L\'evy process is absolutely continuous with respect to
the Lebesgue measure $dz$ on $\mathbb{R}^d$, that is, $\nu(dz)=q(z)dz$, for some measurable function $q:\mathbb{R}^d\rightarrow[0,\infty)$. Let $\{e_1,e_2,...,e_d\}$ be an orthonormal basis of $\mathbb{R}^d$.
Assume that $q$ is a continuous function  and that $q(x)>0$ for  $x=e_i,i=1,2,...,d$ and $x=-\sum_{j=1}^de_i$. Then we can easily see that the intensity measure  $\nu$ satisfies Assumption \ref{assv}.

One can find other mild conditions on $q$ such that the intensity measure $\nu$ satisfies Assumption \ref{assv}, even
for the case that $q$ is not a continuous function.

\end{exmp}

\vskip 0.2cm
\begin{rmk}
One can also easily construct many intensity measures in polar coordinates that satisfy Assumption \ref{assv}.
\end{rmk}

\subsection{Examples for Assumption \ref{ass3}}\label{Sub 4-2}

 A simple example is as follows.
\begin{exmp}
Let $H=\mathbb{R}^d, d\in\mathbb{N}\cup\{+\infty\}$.
Assume that the noise term in (\ref{meq1}) has the form:
\begin{eqnarray*}
  &&\int_{Z_1^c}\si(X(t-), z)\tilde{N}(dz,dt) + \int_{Z_1} \si(X(t-), z)N(dz,dt)\\
  &=&
  \int_{Z_1^c}\si_1(X(t-), z)\tilde{N}_1(dz,dt) + \int_{Z_1}\si_1(X(t-), z)N_1(dz,dt)+dL(t),
\end{eqnarray*}
 here $L$ is a L\'evy process on $H$ satisfying Assumption \ref{assv} (see the  examples in Subsection 4.1), $N_1$ can be any Poisson random measure on $Z$, and $L$ and $N_1$ are independent. Then Assumption \ref{ass3} holds.
\end{exmp}

\vskip 0.2cm
\begin{exmp}
Let $H=\mathbb{R}^d, d\in\mathbb{N}\cup\{+\infty\}$.
 Recall $L=\{L_t=W_{Z_t},t\geq 0\}$ in Example \ref{Prop exm sub}, where $W=\{W_t,t\geq0\}$ is a $Q$-Wiener process on $H$, $Q\in L(H)$ is nonnegative, symmetric, with finite trace and $Ker Q=\{0\}$, $Z=\{Z_t,t\geq0\}$ is a subordinator with L\'evy measure $\rho$ satisfying $\int_{(0,\infty)}(1\wedge s)\rho(ds)>0$,  and $W$ and $Z$ are independent. The intensity measure of $L$ is denoted by $\nu$, which satisfies $S_\nu=H$.
         Denote by $N$ the Poisson random measure corresponding to $L$, and $\tilde{N}$ the associated compensated Poisson random measure. Then
         $$
         L_t=\int_0^t\int_{0<\|z\|_H\leq1}zd\tilde{N}(dz,ds)+\int_0^t\int_{\|z\|_H>1}zdN(dz,ds), t\geq0.
         $$

     Denote by  $L_2(Q^{1/2}(H),H)$ the space of all Hilbert-Schmidt operators from $Q^{1/2}(H)$ to $H$ equipped with the
     Hilbert-Schmidt norm. Assume that $\sigma:H\rightarrow L_2(Q^{1/2}(H),H)$ is continuous, and, for any $h\in H$, $Ker \sigma(h)=0$. Then the driving noise
     \begin{eqnarray*}
     &&\int_0^t\sigma(X(s-))dL_s\\
     &=&
     \int_0^t\int_{0<\|z\|_H\leq1}\sigma(X(s-))zd\tilde{N}(dz,ds)+\int_0^t\int_{\|z\|_H>1}\sigma(X(s-))zdN(dz,ds), t\geq0
     \end{eqnarray*}
     satisfies  Assumption \ref{ass3}.

     This can be seen as follows.
      For any $\hbar,y\in H$ with $\hbar\neq y$ and $\eta>0$, there exists $N\in\mathbb{N}$ such that
     \begin{eqnarray}\label{eq exmp 1}
     \|\hbar-P_{N}\hbar\|_H\leq \frac{\eta}{8}\text{ and }\|y-P_{N}y\|_H\leq \frac{\eta}{8}.
     \end{eqnarray}
     Here, $\{e_1,e_2,...,e_d\}$ is an orthonormal basis of $H$, and  for any $x\in H$, $P_{N}x=\sum_{i=1}^N\langle x,e_i\rangle_H e_i$.

     Let $x_N:=P_N(y-\hbar)$, $l=\sigma(\hbar)^{-1}x_N$ and $q=\hbar+\sigma(\hbar)l=\hbar-P_N\hbar+P_Ny$. Then $\|q-y\|_H\leq \frac{\eta}{4}$.
     By the continuity of $\sigma$ and $S_\nu=H$, it is easy to see that Assumption \ref{ass3} holds with $n=1$.
\end{exmp}

%
\vskip 0.2cm
\begin{exmp}
 Let $H=\mathbb{R}^d, d\in\mathbb{N}\cup\{+\infty\}$, and $\{e_i,i=1,2,...,d\}$ be an orthonormal basis of $H$.
Assume that $\beta_i\in \mathbb{R}\setminus\{0\},i\in\mathbb{N}$, $\{L_i=\{L_i(t),t\geq0\}, i\in\mathbb{N}\}$ be a sequence of i.i.d. one dimensional L\'evy processes with intensity measure $\mu$
on some filtered probability space $(\Omega,\mathcal{F},\{\mathcal{F}_t\}_{t\geq0},\mathbb{P})$, and  $L(t)=\sum_{i=1}^d \beta_iL_i(t)e_i, t\geq0$ defines an $H$-valued L\'evy process, that is,
$$
\sum_{i=1}^d \int_{\mathbb{R}}|\beta_ix_i|^2\wedge 1\mu(dx_i)<\infty.
$$
 Assume further that
\begin{itemize}
  \item[\textbf{(C1)}] there exists $c^+_n>0, c^-_n<0,n\in\mathbb{N}$ such that $\lim_{n\rightarrow\infty}c^-_n=\lim_{n\rightarrow\infty}c^+_n=0$ and $\{c^+_n,c^-_n,n\in\mathbb{N}\}\subseteq S_\mu$.

  \item[\textbf{(C2)}] there exist $\sigma_i: H\rightarrow\mathbb{R}, i\in\mathbb{N}$ such that
  $\sigma_i: H\rightarrow\mathbb{R}$ is continuous and for any $h\in H$,
          $\sigma_i(h)\neq 0.$

\end{itemize}

Denote by $N_i$ the Poisson random measure corresponding to $L_i$, and $\tilde{N}_i$ the associated compensated Poisson random measure. Suppose that the noise term in (\ref{meq1}) is of the form:
\begin{eqnarray*}
  &&\int_0^t\int_{Z_1^c}\si(X(s-), z)\tilde{N}(dz,ds) + \int_0^t\int_{Z_1} \si(X(s-), z)N(dz,ds)\\
  &=&
  \sum_{i=1}^d\int_0^t\beta_i\sigma_i(X(s-))dL_i(s)e_i\\
&=&
\sum_{i=1}^d\Big(
    \int_0^t\int_{0<|z_i|\leq 1}\beta_i\si_i(X(s-))z_i\tilde{N}_i(dz_i,ds)e_i + \int_0^t\int_{|z_i|> 1}\beta_i\si_i(X(s-)) z_iN_i(dz_i,ds)e_i
    \Big).
\end{eqnarray*}
 Then Assumption \ref{ass3} holds.

 This can be seen as follows. For any $h\in H$, set $S_h=\{\beta_i\si_i(h)c^+_ne_i,\beta_i\si_i(h)c^-_ne_i\}_{i,n\in\mathbb{N}}$, and notice that, when $X(s-)=h$, all of the possible jumps generated by the noise term contain $S_h$. Combining \textbf{(C1)} and \textbf{(C2)},
we can  obtain the above result.

\end{exmp}

\subsection{Locally monotone SPDEs}

Under the general  framework as in \cite{BWJ}, we will obtain the irreducibility for coercive and local monotone SPDEs driven by pure jump noise.

The assumptions (H1)-(H4) are very mild so that our result in this subsection ( Proposition 4.1 below)  is applicable to SPDEs such as stochastic reaction-diffusion equations, stochastic semilinear evolution equation, stochastic porous medium equation, stochastic $p$-Laplace equation, stochastic Burgers
type equations, stochastic 2D Navier-Stokes equation, stochastic magneto-hydrodynamic
equations, stochastic Boussinesq model for the B\'enard convection, stochastic 2D magnetic B\'enard problem,  stochastic 3D Leray-$\alpha$
model, stochastic equations of non-Newtonian fluids,  several stochastic Shell Models of
turbulence, and many other stochastic 2D Hydrodynamical systems.

Recall that we consider the following SPDEs.
\begin{eqnarray}\label{a-2}
  &&dX(t)= \cA (X(t)) dt +\int_{Z_1^c}\!\!\!\!\si(X(t-), z)\tilde{N}(dz,dt) + \int_{Z_1}\!\!\!\! \si(X(t-), z)N(dz,dt),\\
  &&X(0)=x.\nonumber
\end{eqnarray}
 Let us formulate the  assumptions on the coefficients $\mathcal{A}$ and $\sigma$. Suppose that there exist constants $\al > 1,\ \be \geq 0,\ \th > 0,\ C > 0$, $F>0$ and a measurable (bounded on balls)  function $\rho: V \rightarrow [0, +\infty)$ such that the following conditions hold for all $v,\ v_1,\ v_2 \in V$:

(H1)  (Hemicontinuity) The map $s \mapsto _{V^*}\langle \mathcal{A}(v_1 + sv_2), v \rangle_V $ is continuous on $\RR$,

(H2) (Local monotonicity)
\begin{eqnarray*}
&&2_{V^*}\langle \mathcal{A}(v_1)-\mathcal{A}(v_2), v_1-v_2\rangle_V +\int_{Z_1^c}\|\sigma(v_1,z) - \sigma(v_2,z)\|_H^2 \nu(dz) \\
&\leq& (C+\rho(v_2))\|v_1 - v_2\|_H^2,
\end{eqnarray*}

(H3) (Coercivity)
\[
2_{V^*}\langle \mathcal{A}(v), v \rangle_V + \th \|v\|_V^{\al} \leq F +C\|v\|_H^2,
\]

(H4) (Growth)
\[
\|\mathcal{A}(v)\|_{V^*}^{\frac{\al}{\al-1}} \leq (F + C\|v\|_V^{\al})\big(1 + \|v\|_H^{\be}\big).
\]
\vskip 0.2cm
The following well-posedness was proved in \cite[Theorem 1.2]{BWJ}.
\begin{lem}\label{lem monontone}
  Suppose that conditions (H1)-(H4) hold, and there exists a constant $\gamma<\frac{\theta}{2\beta}$
such that for all $v\in V$
\begin{eqnarray*}
\int_{Z_1^c}\|\sigma(v,z)\|_H^2\nu(dz)\leq F+C\|v\|_H^2+\gamma\|v\|_V^\alpha;
\end{eqnarray*}
\begin{eqnarray*}
\int_{Z_1^c}\|\sigma(v,z)\|_H^{\beta+2}\nu(dz)
\leq
F^{\frac{\beta+2}{2}}+C\|v\|_H^{\beta+2};
\end{eqnarray*}
\begin{eqnarray*}
\rho(v)\leq C(1+\|v\|_V^\alpha)(1+\|v\|_H^{\beta}).
\end{eqnarray*}
Then for any $x\in H$, (\ref{a-2}) has a unique solution $X^x=(X^x(t),t\geq0)$.
\end{lem}

\vskip 0.2cm
Now, we state the main result in this  subsection.
\begin{prop}\label{Prop monotone}
  Under the same assumptions of Lemma \ref{lem monontone}, assume that for any fixed $z\in Z_1$, $\sigma(\cdot,z): H\rightarrow H$ is continuous, and that the driving  noise term satisfies Assumption \ref{ass3}, then the solution $\{X^x,x\in H\}$ to equation (\ref{a-2}) is irreducible in $H$.
\end{prop}

\begin{proof}
By Lemma \ref{lem monontone} and the fact that for any fixed $z\in Z_1$, $\sigma(\cdot,z): H\rightarrow H$ is continuous, it is classical that $\{X^x,x\in H\}$ forms a strong Markov process on $H$. Therefore, Assumption \ref{ass1} is satisfied.

  Applying Theorem \ref{thmmulti}, we see that the proof of this proposition will be  complete once we prove that  Assumption \ref{ass2} is satisfied, which we will do in the rest of the proof.

The proof is divided into two steps.

\textbf{{Step 1}}. Consider (\ref{meq1-Zhai}) with $m=1$, that is
  \begin{eqnarray}\label{eq mono 1}
  &&dX_1(t)= \cA (X_1(t)) dt +\int_{Z_1^c}\!\!\!\!\si(X_1(t-), z)\tilde{N}(dz,dt), \nonumber\\
  &&X_1(0)=x.
\end{eqnarray}
 By \cite[Theorem 1.2]{BWJ}, for any $x\in H$, (\ref{eq mono 1}) has a unique solution $X^x_1=(X^x_1(t),t\geq0)$.

For any $h , \tilde{h} \in H$, applying the It\^{o} formula, we have
\begin{eqnarray*}
&&e^{-\int_0^t (C + \rho(X_1^h(s)))ds}\| X_1^{\tilde{h}}(t) - X_1^h(t)\|_H^2 - \|\tilde{h}-h\|_H^2 \\
&\leq&
   \int_0^t e^{-\int_0^s (C + \rho(X_1^h(r)))dr}\Big(2_{V^*}\langle \mathcal{A}(X_1^{\tilde{h}}(s))-\mathcal{A}(X_1^h(s)), X_1^{\tilde{h}}(s)-X_1^h(s)\rangle_V \\
&&-
  (C+\rho(X_1^h(s)))\| X_1^{\tilde{h}}(s) - X_1^h(s)\|_H^2 \Big)ds \\
&&+
   2\int_0^t\!\! \int_{Z_1^c} e^{-\int_0^s\! (C + \rho(X_1^h(r)))dr} \langle \sigma(X_1^{\tilde{h}}(s-),z) \!-\! \sigma(X_1^h(s-),z), X_1^{\tilde{h}}(s-) \!-\! X_1^h(s-) \rangle_H \tilde{N}(dz,ds)\\
&&+
    \int_0^t\!\! \int_{Z_1^c} e^{-\int_0^s (C + \rho(X_1^h(r)))dr} \|\sigma(X_1^{\tilde{h}}(s-),z) - \sigma(X_1^h(s-),z)\|_H^2 N(dz,ds)\\
&\leq&
    2\int_0^t\!\! \int_{Z_1^c} e^{-\int_0^s (C + \rho(X_1^h(r)))dr} \langle \sigma(X_1^{\tilde{h}}(s-),z) - \sigma(X_1^h(s-),z), X_1^{\tilde{h}}(s-) - X_1^h(s-) \rangle_H \tilde{N}(dz,ds)\\
&&+
   \int_0^t\!\! \int_{Z_1^c} e^{-\int_0^s (C + \rho(X_1^h(r)))dr} \|\sigma(X_1^{\tilde{h}}(s-),z) - \sigma(X_1^h(s-),z)\|_H^2 \tilde{N}(dz,ds).
\end{eqnarray*}
Assumption (H2) has been used for the last inequality. Applying stochastic Gronwall's inequality, see \cite[Lemma 3.7]{XZ}, we deduce that for any $0<q<p<1$ and any $T>0$,
\begin{eqnarray}\label{eq mono pro}
\EE\Big[\Big(\sup_{0\leq t \leq T}e^{-\int_0^{t} (C + \rho(X_1^h(s)))ds}\| X_1^{\tilde{h}}(t) - X_1^h(t)\|_H^2\Big)^q\Big]
\leq
(\frac{p}{p-q})\|\tilde{h}-h\|_H^{2q}.
\end{eqnarray}
Define the stopping time
$$\tau_1^h := \inf \{t \geq 0: e^{-\int_0^{t} (C + \rho(X_1^h(s)))ds} \leq 1/2\}.$$
 Since $X^h_1\in L^\alpha_{loc}([0,\infty),V)\cap D([0,\infty),H), \mathbb{P}$-a.s. (see \cite[Definition 1.1]{XZ}), $\PP(\tau_1^h >0) = 1$. Therefore,
\begin{eqnarray}\label{eq mono stop}
\lim_{T\searrow0}\PP(\tau_1^h <T) = 0.
\end{eqnarray}
By Chebyshev's inequality, for any $\eta>0$,
\ba
&&\PP\Big(\sup_{0\leq t \leq T}\| X_1^{\tilde{h}}(t) -X_1^h(t)\|_H > \frac{\et}{2}\Big) \notag \\
&\leq&\PP\Big(\sup_{0\leq t \leq T}\| X_1^{\tilde{h}}(t) -X_1^h(t)\|_H > \frac{\et}{2},  \tau_1^h \geq T\Big) + \PP(\tau_1^h < T) \notag \\
&\leq&\PP\Big(\sup_{0\leq t \leq T}e^{-\int_0^{t} (C + \rho(X_1^h(s)))ds}\| X_1^{\tilde{h}}(t) -X_1^h(t)\|_H^2 > \frac{\et^2}{8},  \tau_1^h \geq T\Big) + \PP(\tau_1^h < T)  \notag \\
&\leq& (\frac{p}{p-q})\|\tilde{h}-h\|_H^{2q}\frac{(8)^q}{(\eta^2)^q} + \PP(\tau_1^h < T).
\ea

Notice that
\begin{eqnarray}\label{eq mon 2}
&&\PP(\sup_{0\leq t \leq T}\|X_1^{\tilde{h}}(t)-h\|_H > \et)\nonumber\\
&\leq&
\PP(\sup_{0\leq t \leq T}\|X_1^{h}(t)-h\|_H > \frac{\et}{2})
+
\PP(\sup_{0\leq t \leq T}\|X_1^{\tilde{h}}(t) -X_1^h(t)\|_H > \frac{\et}{2}),
\end{eqnarray}
and $X^h_1\in D([0,\infty),H), \mathbb{P}$-a.s. implies that
\begin{eqnarray}\label{eq mon 2-zhai00}
\lim_{T\searrow 0}\PP(\sup_{0\leq t \leq T}\|X_1^{h}(t)-h\|_H > \frac{\et}{2})=0.
\end{eqnarray}
Combining (\ref{eq mono stop})-(\ref{eq mon 2-zhai00}) together we see that  there exist $T_0=T_0(h,\eta)>0$ and $\epsilon_0=\epsilon_0(h,\eta)>0$ small enough such that
\begin{eqnarray}\label{eq mon 2-1}
\sup_{\tilde{h}\in B(h,\epsilon_0)}\PP(\sup_{0\leq t \leq T_0}\|X_1^{\tilde{h}}(t)-h\|_H > \et)
\leq
\frac{1}{4}.
\end{eqnarray}
Therefore,
\begin{eqnarray}\label{eq mon 3}
\inf_{\tilde{h}\in B(h,\epsilon_0)}\PP(\sup_{0\leq t \leq T_0}\|X_1^{\tilde{h}}(t)-h\|_H \leq \et)
\geq
\frac{3}{4}.
\end{eqnarray}

\textbf{{Step 2}}.  Recall $\tau_1^1=\inf\{t\geq 0: N(Z_1,t)=1\}$ as  defined in (\ref{stop time 1}). Notice that
$\{X^x(t),t\in[0,\tau_1^1)\}$ coincides with $\{X_1^x(t),t\in[0,\tau_1^1)\}$.
By the independence of $X^x_1$ and $\tau_1^1$ (see Proposition \ref{Prop 3}), we have
\begin{eqnarray}\label{eq mon 4}
&&\inf_{\tilde{h}\in B(h,\epsilon_0)}\PP(\sup_{0\leq t \leq T_0}\|X^{\tilde{h}}(t)-h\|_H \leq \et)\nonumber\\
&\geq&
\inf_{\tilde{h}\in B(h,\epsilon_0)}\PP(\sup_{0\leq t \leq T_0}\|X^{\tilde{h}}_1(t)-h\|_H \leq \et,\tau_1^1>2T_0)\nonumber\\
&\geq&
\inf_{\tilde{h}\in B(h,\epsilon_0)}\PP(\sup_{0\leq t \leq T_0}\|X^{\tilde{h}}_1(t)-h\|_H \leq \et)\PP(\tau_1^1>2T_0)\nonumber\\
&\geq&\frac{3}{4}\PP(\tau_1^1>2T_0)>0.
\end{eqnarray}
For the last inequality, we have used the fact that $\tau_1^1$  has the exponential distribution with parameter $\nu(Z_1)<\infty$.
(\ref{eq mon 4}) implies  Assumption \ref{ass2}, completing the proof.
\end{proof}

\subsection{Nonlinear Schr\"odinger equations}

The nonlinear Schr\"odinger equation (NLS) is a fundamental model describing wave propagation that appears in various fields such as nonlinear optics, nonlinear water propagation, quantum physics, Bose-Einstein condensate, plasma physics and molecular biology.

In this subsection, we study the irreducibility of stochastic NLS driven by additive L\'evy noise. Without further notice, all the $L^p$ spaces in this subsection are referred as spaces of complex-values functions.

Consider (\ref{meq1-additive case}) with  $H = L^2(\RR^d)$, $d\in\mathbb{N}$, $\cA(u) = i[\Delta u- \lambda |u|^{\al -1}u]$, where $\lambda \in \{-1, \textcolor[rgb]{0.00,0.07,1.00}{0,} 1\}$, and $1 < \al< 1+\frac{4}{d}$. Now consider NLS with additive noise, that is,
\begin{eqnarray}\label{a-3}
  &&dX(t)= \cA (X(t)) dt +\int_{0<\|z\|_H\leq 1}\!\!\!\! z\tilde{N}(dz,dt) + \int_{\|z\|_H>1}\!\!\!\! zN(dz,dt),\\
  &&X(0)=x.\nonumber
\end{eqnarray}

We say a pair $(p, r)$ is admissible if $p, r \in [2,\infty]$ and $(p,r,d) \ne (2,\infty ,2)$ satisfying $\frac{2}{p}+\frac{d}{r}=\frac{d}{2}$.
The following result provides the existence and uniqueness of the solution of the stochastic NLS (\ref{a-3}) whose proof was given in \cite{WZ}.
\vskip 0.2cm
\begin{prop}\label{prop 4}
  Let $p\geq 2$, $1<\al<1+\frac{4}{d}$, $r=\al +1$  such that  $(p,r)$ is an admissible pair. For any $\hbar\in  H$,
there exists a unique global mild solution  $X^\hbar=(X^\hbar(t),t\geq 0)$ of (\ref{a-3}) satisfying
\[
X^\hbar\in D([0,\infty); H)\cap L^p_{loc}(0,\infty;L^r(\RR^d)),\ \mathbb{P}\text{-a.s.}
\]
\end{prop}

\vskip 0.2cm
Here is the result for the irreducibility of the solution.
\begin{prop}\label{Pro 4.3}

If the driving  noise satisfies Assumption \ref{assv}, then the solution $\{X^x,x\in H\}$ of (\ref{a-3}) is irreducible in $H$.
Moreover, for the linear Schr\"odinger equation case, i.e., $\lambda=0$, there exists at most one invariant measure.

\end{prop}

\begin{proof}
For the case of $\lambda=0$, we have
$$
\|X^{\hbar_1}(t)-X^{\hbar_2}(t)\|_H\leq \|\hbar_1-\hbar_2\|_H, \ \forall t\geq0,\ \hbar_1,\hbar_2\in H.
$$
This together with the irreducibility implies  the uniqueness of the invariant measure (if it exists); see  Theorem 2 in \cite{Kapica}.
In the following, we will prove irreducibility.

It is classical that the solution $\{X^\hbar,\hbar\in H\}$ of (\ref{a-3}) forms a strong Markov process on $H$.
  Therefore, Assumptions \ref{ass1} holds.
To apply Theorem \ref{thmmulti-additive case} to conclude the proof, it only remains to show that  Assumption \ref{ass2} holds.
From the proof of Proposition \ref{Prop monotone} we see that  Assumption \ref{ass2} is implied by   the following results.

 For any $h \in  H$ and $\et > 0$, there exist $T_0=T_0(h,\eta)>0$ and $\epsilon_0=\epsilon_0(h,\eta)>0$ small enough such that
\begin{equation}\label{NLS-0}
\sup_{\tilde{h} \in B(h, \eps_0)}\PP\Big(\sup_{0\leq t \leq T_0}\| X_1^h(t) - X_1^{\tilde{h}}(t)\|_{ H}> \frac{\et}{2}\Big) < \frac{1}{2},
\end{equation}
where $X^\hbar_1$ is the mild solution of the following equation:
\begin{eqnarray}\label{a-4}
  &&dX^\hbar_1(t)= \cA (X^\hbar_1(t)) dt +\int_{0<\|z\|_H\leq 1}\!\!\!\! z\tilde{N}(dz,dt),\\
  &&X^\hbar_1(0)=\hbar.\nonumber
\end{eqnarray}
 We now prove (\ref{NLS-0}). For $t>0$,  set
\begin{align}
Y_{t}:=L^{\infty}(0,t; H)\cap L^p(0,t;L^r(\RR^d)),
\end{align}
and for $u\in Y_t$,
  \begin{align}
   \|u\|_{Y_{t}}:=\sup_{s\in[0,t]}\|u(s)\|_{ H}+\Big(\int_0^t\|u(s)\|_{L^r(\RR^d)}^p\d s\Big)^{\frac1p}.
  \end{align}
Let $\theta:\mathbb{R}_+\rightarrow[0,1]$ be a non-increasing $C_0^{\infty}$ function such that
  $1_{[0,1]}\leq \theta \leq 1_{[0,2]}$ and $\inf_{x\in\RR_+}\theta'(x)\geq-2$. For the fixed $h\in  H$, set $R=\|h\|_{ H}+2$ and $\theta_R(\cdot)=\theta(\frac{\cdot}{R})$.

For any $\hbar\in  H$, let $Z^\hbar_R$ be the solution of the trancated stochastic Schr\"odinger equation:
\begin{eqnarray}\label{a-5}
  &&dZ^\hbar_R(t)= i[\Delta Z^\hbar_R(t)- \lambda \theta_R(\|Z^\hbar_R\|_{Y_t})|Z^\hbar_R(t)|^{\al-1}Z^\hbar_R(t)] dt +\int_{0<\|z\|_H\leq 1}\!\!\!\! z\tilde{N}(dz,dt),\\
  &&Z^\hbar_R(0)=\hbar.\nonumber
\end{eqnarray}
By \cite[Propositions 2.2 and 3.1]{BLZ}, for any $T>0$, we have
\ba
\|Z_R^h - Z_R^{\tilde{h}}\|_{Y_{T}}
\leq
C\|h-\tilde{h}\|_{ H} +CT^{1-\frac{(\al-1)d}{4}}R^{\al-1}\|Z_R^h - Z_R^{\tilde{h}}\|_{Y_{T}}.
\ea
For any $\eta>0$, we can choose $\widetilde{\epsilon}\in(0,\frac{\eta}{4}\wedge 1]$ and $\widetilde{T}>0$  small enough such that
\[2C\widetilde{\epsilon}\leq \frac{1}{2}\wedge \frac{\eta}{4}
   \text{ and }
   \quad C\widetilde{T}^{1-\frac{(\al-1)d}{4}}R^{\al-1} \leq \frac{1}{2}.\]
 Then, for any $\tilde{h}\in B(h,\widetilde{\epsilon})$,
\begin{equation}\label{NLS-1}
\|Z_R^h - Z_R^{\tilde{h}}\|_{Y_{\widetilde{T}}}  \leq 2C\|h-\tilde{h}\|_{ H}\leq \frac{1}{2}\wedge \frac{\eta}{4},
\end{equation}
here, for any $\epsilon>0$, $B(h,\epsilon)=\{\hbar\in  H: \|\hbar-h\|_{ H}<\epsilon\}$.

Define $\tau=\inf\{s \geq 0: \| Z^h_R \|_{_{Y_s}} > \|h\|_{ H}+1/2\}$. Then $\PP(\tau>0)=1$, which implies that
there exists $T_0\in (0, \widetilde{T}]$ such that
\begin{eqnarray}\label{eq XDE ST 00}
\PP(\tau>T_0)\geq 11/12.
\end{eqnarray}
Combining this with (\ref{NLS-1}), for any $\tilde{h}\in B(h,\widetilde{\epsilon})$,
\begin{eqnarray}\label{eq XDE 00}
\|Z_R^{\tilde{h}}\|_{Y_{T_0}}  \leq \|h\|_{ H}+1<R \text{ on }\ \{\tau>T_0\}\ \PP\text{-a.s..}
\end{eqnarray}

Let us define   $\tau^\hbar_R = \inf\{s \geq 0: \| Z^\hbar_R \|_{_{Y_s}} > R\}$. Then, for any $\tilde{h}\in B(h,\widetilde{\epsilon})$,
\begin{eqnarray}\label{eq XDE 01}
\PP(\tau^{\tilde{h}}_R>0)=1 \text{ and }Z^{\tilde{h}}_R(t)=X^{\tilde{h}}_1(t)\text{ on }t\in[0,\tau^{\tilde{h}}_R)\ \PP\text{-a.s.}.
\end{eqnarray}

Note that (\ref{eq XDE 00}) implies that, for any $\tilde{h}\in B(h,\widetilde{\epsilon})$,
\begin{eqnarray}\label{eq XDE 02}
\PP(\tau\wedge\tau^{\tilde{h}}_R>T_0)=\PP(\tau>T_0).
\end{eqnarray}

Combining (\ref{NLS-1})--(\ref{eq XDE 02}) together, we deduce that, for any $\tilde{h}\in B(h,\widetilde{\epsilon})$,
\ba
&&\PP\Big(\sup_{0\leq t \leq T_0}\| X_1^h(t) - X_1^{\tilde{h}}(t)\|_{ H}> \frac{\et}{2}\Big) \notag\\
&\leq&\PP\Big(\sup_{0\leq t \leq T_0}\| Z_R^h(t) - Z_R^{\tilde{h}}(t)\|_{ H}> \frac{\et}{2},  \tau\wedge\tau^{\tilde{h}}_R > T_0\Big) + \PP(\tau\wedge\tau^{\tilde{h}}_R\leq T_0) \notag \\
&=&  \PP(\tau \leq T_0) \leq \frac{1}{12}.
\ea

 This completes the proof.
\end{proof}

\subsection{Singular SDEs}\label{sub section 4.5}
      Let $L=(L_t)_{t\geq0}$ be a L\'evy process on $\mathbb{R}^d, d\in\mathbb{N}$, and denote its intensity measure by $\nu$.
      To state the condition on $\nu$, for $\alpha\in(0,2)$, denote by $\mathbb{L}_{non}^\alpha$ the space of all non-degenerate $\alpha$-stable L\'evy measure $\nu^{(\alpha)}$; that is,
      $$
      \nu^{(\alpha)}(A)=\int_0^\infty(\int_{\mathbb{S}^{d-1}}\frac{1_{A}(r\theta)\vartheta(d\theta)}{r^{1+\alpha}})dr,\ \ A\in\mathcal{B}(\mathbb{R}^d),
      $$
      where $\vartheta$ is a finite measure over the unit sphere $\mathbb{S}^{d-1}$ in $\mathbb{R}^d$ with
      \begin{eqnarray}\label{eq 1.4}
      \inf_{\theta_0\in \mathbb{S}^{d-1}}\int_{\mathbb{S}^{d-1}}|\theta_0\cdot\theta|\vartheta(d\theta)>0.
      \end{eqnarray}

For $R>0$, denote by $B_R$ the closed ball in $\mathbb{R}^d$ centered at the origin with radius
$R$. We assume that there are $\nu_1,\nu_2\in \mathbb{L}_{non}^\alpha$, so that
\begin{eqnarray}\label{eq ZXC IM}
\nu_1(A)\leq\nu(A)\leq\nu_2(A)\ \text{for}\ A\in\mathcal{B}(B_1).
\end{eqnarray}
In \cite{CZZ}, the authors call L\'evy processes with intensity measure satisfying (\ref{eq ZXC IM})
non-degenerate $\al$-stable-like L\'evy process. The L\'evy measure $\nu$ could be singular with respect to
the Lebesgue measure on $\mathbb{R}^d$ and its support could be a proper subset of $\mathbb{R}^d$.

For a Borel measurable drift  $b(\cdot):\mathbb{R}^d\rightarrow\mathbb{R}^d$ and diffusion matrix $\sigma(\cdot):\mathbb{R}^d\rightarrow\mathbb{R}^d\otimes \mathbb{R}^d$, consider the following SDE
\begin{eqnarray}\label{eq ZXC}
dX_t&=&b(X_t)dt+\sigma(X_{t-})dL_t\nonumber\\
&=&
   b(X_t) dt +\int_{0<|z|\leq 1}\si(X_{t-})z\widetilde{N}(dz,dt)+\int_{|z|> 1}\si(X_{t-})z N(dz,dt).
\end{eqnarray}
Here $N$ and $\widetilde{N}$ are the Poisson random measure and compensated Poisson random measure  associated with $L$, respectively.
In a recent paper \cite{CZZ}, the authors established the following well posedness of the SDE (\ref{eq ZXC}).

\begin{lem}\label{lem singular}
 Assume that $\nu$ satisfies (\ref{eq ZXC IM}) with $\alpha\in(0,2)$. Assume that there
are constants $\beta\in(1-\alpha/2,1]$ and $\Lambda>0$ so that for all $x,y,\xi\in\mathbb{R}^d$,
\begin{eqnarray}\label{eq ZXC con1}
&& |b(x)|\leq \Lambda\text{ and }|b(x)-b(y)|\leq \Lambda|x-y|^\beta,\\
&& \Lambda^{-1}|\xi|\leq |\sigma(x)\xi|\leq \Lambda|\xi|\text{ and }\|\sigma(x)-\sigma(y)\|\leq\Lambda|x-y|.
\end{eqnarray}
Then, there is a unique strong solution $X^x=(X^x(t),t\geq0)$ to (\ref{eq ZXC}) for any initial data $x\in\mathbb{R}^d$.
\end{lem}

We are concerned with  the irreducibility of the solutions $\{X^x,x\in\mathbb{R}^d\}$ on $\mathbb{R}^d$.
To obtain the irreducibility, we introduce the following conditions. Let $\{e_i\}_{i=1,2,...,d}$ be an orthonormal basis of $\mathbb{R}^d$.
\begin{itemize}
  \item[(I)]There exist $n\in\mathbb{N}$,
                $f_1,f_2,...,f_n\in \mathbb{S}^{d-1}$, and $\kappa\in(0,1]$, such that
                $\{f_1,f_2,...,f_n\}\subset S_{\vartheta}$, and
  for any $x\in\mathbb{R}^d$, $\inf_{y\in  \mathbb{S}^{d-1}}\sup_{i=1,2,...,n}\frac{\langle\sigma(x)f_i,y\rangle}{|\sigma(x)f_i|}\geq \kappa$.
\end{itemize}

\begin{prop}\label{prop 4.4}
Under the same assumptions of Lemma \ref{lem singular}, and assume that (I) holds, the solutions to (\ref{eq ZXC}) is irreducible in $\RR^d$.
\end{prop}\label{prop ZXC 02}

%
%

\begin{proof}
We will apply Theorem 2.1 to get the irreducibility.  First we verify  Assumption \ref{ass2}.

 Removing the big jumps in (\ref{eq ZXC}), consider the following SDE:
\begin{equation}\label{sde_1}
dX_1(t)= b(X_1(t)) dt +\int_{0<|z|\leq 1}\si(X_1(t-))z\widetilde{N}(dz,dt), \qquad X_1(0)=x.
\end{equation}
We will prove that for any $h \in \RR^d$ and $\et > 0$, there exist $T_0=T_0(h,\eta)>0$ and $\epsilon_0=\epsilon_0(h,\eta)>0$ small enough such that
\begin{equation}\label{keyineq}
\sup_{\tilde{h} \in B(h, \eps_0)}\PP\Big(\sup_{0\leq t \leq T_0}| X_1^h(t) - X_1^{\tilde{h}}(t)|> \frac{\et}{2}\Big) < \frac{1}{2}.
\end{equation}

We now  fix $h\in \RR^d$ and set  $R=|h|+2$. Let $\theta:\mathbb{R}_+\rightarrow[0,1]$ be a non-increasing $C_0^{\infty}$ function such that
  $1_{[0,R]}\leq \theta \leq 1_{[0,R+2]}$. Let $b_\theta(x)=b(x)\theta(x)$. For any $\zeta\in(0,1]$, define
\begin{eqnarray*}
\sigma_\zeta(x)=\left\{
           \begin{array}{lll}
               \!\!\! \sigma(x+h), &\mbox{$|x|\leq\zeta/2$,}\\
               \!\!\!  \frac{2(\zeta-|x|)}{\zeta}\sigma(\frac{\zeta x}{2|x|}+h)
                        +
                        \frac{2(|x|-\zeta/2)}{\zeta}\sigma(h), &\mbox{$\zeta/2<|x|\leq\zeta$,}\\
               \!\!\! \sigma(h), &\mbox{$|x|>\zeta$.}
           \end{array}
         \right.
\end{eqnarray*}

For $\zeta_0\in(0,1]$ small enough, by the proof of \cite[Theorem 1.1]{CZZ} and
applying \cite[Theorem 4.1]{CZZ}, for any $\hbar\in\RR^d$, the following SDE admits a unique strong solution
\begin{eqnarray*}
Y_t^\hbar=\hbar+\int_0^tb_\theta(Y_s^\hbar+h)ds+\int_0^t\int_{0<|z|\leq 1}\si_{\zeta_0}(Y_{s-}^\hbar)z\widetilde{N}(dz,dt),\ t\geq0.
\end{eqnarray*}
Moreover, from the proof of \cite[Theorem 4.1]{CZZ} we see  that for any $T>0$ there exists a constant $C_T>0$ such that for any $\hbar_1,\hbar_2\in \RR^d$
\begin{eqnarray}\label{eq X and Y 01}
\mathbb{E}(\sup_{t\in[0,T]}|Y_t^{\hbar_1}-Y_t^{\hbar_2}|^2)
\leq
C_T|\hbar_1-\hbar_2|^2.
\end{eqnarray}

Define
$$
\tau^\hbar=\inf\{t>0:|Y_t^\hbar|\geq{\zeta_0}/2\}.
$$
Then for any $\tilde{h}\in B(h, {\zeta_0}/8)$,
\begin{eqnarray}\label{eq tau ex}
\PP(\tau^{\tilde{h}-h}>0)=1,
\end{eqnarray}
\begin{eqnarray}\label{eq X and Y}
X_1^{\tilde{h}}(t)=h+Y_t^{\tilde{h}-h} \text{ on } t\in[0,\tau^{\tilde{h}-h})\ \PP\text{-a.s.},
\end{eqnarray}
and
\begin{eqnarray}\label{eq X and Y}
X_1^{{h}}(t)-X_1^{\tilde{h}}(t)=Y_t^{0}-Y_t^{\tilde{h}-h} \text{ on } t\in[0,\tau^0\wedge\tau^{\tilde{h}-h})\ \PP\text{-a.s.}.
\end{eqnarray}

Define $\tilde{\tau}^0=\inf\{t>0:|Y_t^0|\geq{\zeta_0}/16\}$. Then $\PP(\tilde{\tau}^0>0)=1$, which implies that there
exists $T_0\in(0,1]$ such that
\begin{eqnarray}\label{eq tau ex 1}
\PP(\tilde{\tau}^0>T_0)\geq 11/12.
\end{eqnarray}
Combining the inequality above with (\ref{eq X and Y 01}) and the Chebyshev inequality,
\begin{eqnarray*}
&&\PP(\sup_{t\in[0,T_0]}|Y_t^{\tilde{h}-h}|\leq \zeta_0/8)\nonumber\\
&\geq&
\PP(\{\sup_{t\in[0,T_0]}|Y_t^{\tilde{h}-h}-Y_t^0|\leq \zeta_0/16\} \cap\{\sup_{t\in[0,T_0]}|Y_t^{0}|\leq \zeta_0/16\})\nonumber\\
&\geq&
1-\PP(\sup_{t\in[0,T_0]}|Y_t^{\tilde{h}-h}-Y_t^0|> \zeta_0/16)
-\PP(\sup_{t\in[0,T_0]}|Y_t^{0}|> \zeta_0/16)\nonumber\\
&\geq&
1-\frac{16^2}{\zeta_0^2}C_{T_0}|\tilde{h}-h|^2-\PP(\tilde{\tau}^0\leq T_0)\nonumber\\
&\geq& 11/12-\frac{16^2}{\zeta_0^2}C_{T_0}|\tilde{h}-h|^2.
\end{eqnarray*}
Hence, there exists $\epsilon_1>0$ such that
\begin{eqnarray*}
\inf_{\tilde{h}\in B(h,\epsilon_1)}\PP(\sup_{t\in[0,T_0]}|Y_t^{\tilde{h}-h}|\leq \zeta_0/8)
\geq
 10/12,
\end{eqnarray*}
which implies that
\begin{eqnarray*}
\inf_{\tilde{h}\in B(h,\epsilon_1)}\PP(\tau^{\tilde{h}-h}> T_0)
\geq
 10/12.
\end{eqnarray*}
Combining this inequality with (\ref{eq tau ex 1}), we further have  that for any $\tilde{h}\in B(h,\epsilon_1)$
\begin{eqnarray}\label{eq es 00}
\PP(\tau^0\wedge\tau^{\tilde{h}-h}> T_0)
&\geq&
 1-\PP(\tau^0\leq T_0)-\PP(\tau^{\tilde{h}-h}\leq T_0)\nonumber\\
 &\geq&
 \PP(\tau^{\tilde{h}-h}> T_0)-\PP(\tilde{\tau}^0\leq T_0)\nonumber\\
 &\geq&\frac{3}{4}.
\end{eqnarray}
For the second inequality, we have used $\{\tau^0\leq T_0\}\subseteq\{\tilde{\tau}^0\leq T_0\}$.

For any $\eta>0$, let $\epsilon_0=\epsilon_1\wedge \sqrt{\frac{\eta^2}{32C_{T_0}}}\wedge \frac{\eta}{4}$. Then,
by (\ref{eq X and Y 01}), (\ref{eq X and Y}), the Chebyshev inequality, and (\ref{eq es 00}),  for any $\tilde{h}\in B(h,\epsilon_0)$,
\begin{eqnarray*}
&&\PP\Big(\sup_{0\leq t \leq T_0}| X_1^h(t) - X_1^{\tilde{h}}(t)|> \frac{\et}{2}\Big)\\
&\leq&
\PP\Big(\{\sup_{0\leq t \leq T_0}| X_1^h(t) - X_1^{\tilde{h}}(t)|> \frac{\et}{2}\}
\cap
\{\tau^0\wedge\tau^{\tilde{h}-h}> T_0\}
\Big)
+
\PP(\tau^0\wedge\tau^{\tilde{h}-h}\leq T_0)\\
&\leq&
\PP\Big(\sup_{0\leq t \leq T_0}|Y_t^{0}-Y_t^{\tilde{h}-h}|> \frac{\et}{2}
\Big)
+
\frac{1}{4}\\
&\leq&
\frac{4}{\eta^2}C_{T_0}|\tilde{h}-h|^2+\frac{1}{4}\\
&\leq&
3/8.
\end{eqnarray*}

The proof of (\ref{keyineq}) is complete. Now following the  similar arguments as that in the proof of Proposition \ref{Prop monotone}, we see that Assumption \ref{ass2} holds.

Assumption \ref{ass1} follows from  Lemma \ref{lem singular}.  We now verify Assumption \ref{ass3}.

Note that $\{lf_i,l\in(0,1], i=1,2,...,n\}\subseteq S_\nu$.

For any $\hbar\neq y\in \mathbb{R}^d$ and $\eta>0$, set $q_0=\hbar$. Choose $i_0\in\{1,2,...,n\}$ such that
$$\varpi_0
:=
\frac{\langle\sigma(q_0)f_{i_0},y-q_0\rangle}{|\sigma(q_0)f_{i_0}||y-q_0|}
=
\sup_{i=1,2,...,n}\frac{\langle\sigma(q_0)f_i,y-q_0\rangle}{|\sigma(q_0)f_i||y-q_0|}.
$$
Then $\varpi_0\in[\kappa,1]$. Let $\varrho=|y-q_0|$ and $\theta$ be such that $\cos \theta=\varpi_0$. Define
$$g(r)=(\varrho-r\cos \theta )^2+(r\sin \theta)^2=\varrho^2-2r\varrho\varpi_0+r^2\leq \varrho^2-2r\varrho\kappa+r^2,\ \ r\geq0.$$
Take $r_0\in(0,|\sigma(q_0)f_{i_0}|]\supset(0,\Lambda^{-1}]$ such that $g(r_0)=\inf_{r\in(0,|\sigma(q_0)f_{i_0}|]}g(r)\leq\inf_{r\in(0,\Lambda^{-1}]}g(r)$. Since $\varpi_0\in[\kappa,1]$,
if $\varrho\kappa>\Lambda^{-1}$,  $g(r_0)\leq  g(\Lambda^{-1})\leq\varrho^2-\Lambda^{-2}$; if $\varrho\kappa\in(0,\Lambda^{-1}]$, $g(r_0)= g(\varrho\kappa) =\varrho^2(1-\kappa^2)$.

Now let $q_1=q_0+\frac{\sigma(q_0)f_{i_0}}{|\sigma(q_0)f_{i_0}|}r_0$ and $l_1=\frac{f_{i_0}}{|\sigma(q_0)f_{i_0}|}r_0$. Then $|q_1-y|^2=g(r_0)$. Recursively, we can construct $q_m,l_m,m\geq 2$ until that $|q_m-y|\leq \frac{\eta}{8}$.
Since $\sigma$ is continuous and $\{lf_i,l\in(0,1], i=1,2,...,n\}\subseteq S_\nu$, then Assumption \ref{ass3} holds.

The proof of the proposition is complete.
\end{proof}

\vskip 0.6cm
\section{Applications II: ergodicity}

In Proposition \ref{Pro 4.3}, we obtained the uniqueness of invariant measure for the linear Schr\"odinger equation. In this section, we will see many more interesting examples for which we obtain the ergodicity.
We consider the following stochastic evolution inclusion in a separable Hilbert space $H$ driven by pure jump noise:
\begin{eqnarray}\label{eq Ergodicity 1}
  &&dX(t)-\cA (X(t)) dt \ni dL(t),\ t>0,\\
  &&X(0)=X_0.\nonumber
\end{eqnarray}

Let $H$ be a separable
Hilbert space with dual $H^*$. Suppose that there is another Hilbert space $S$ embedded densely and compactly into $H$.
We thus have a Gelfand triple
$$
S\subset H\cong H^*\subset S^*
$$
and it holds that
$$
_{S^*}\langle v,u\rangle_{S}=\langle v,u\rangle_{H}, \text{ whenever }u\in S, v\in H.
$$
Let $i_S:S\rightarrow S^*$ denote the Riesz map of $S$. We note that the scalar product $\langle \cdot,\cdot\rangle_{S}$ defines a bilinear, $S$-bounded,
$S$-coercive form on $H$. By the Lax-Milgram Theorem there is a linear, positive definite, self-adjoint operator
$T:D(T)\subset H\rightarrow H$ with $D(T^{1/2}) = S$ and $\langle  T^{1/2}u,T^{1/2}v\rangle_{H}=\langle u,v\rangle_{S}$.
We define $J_n=(1+\frac{T}{n})^{-1}, n\in\mathbb{N}$, to be
the resolvent associated to $T$ and $T_n=TJ_n=n(1-J_n)$ to be the Yosida approximation of $T$.

Suppose that $\mathcal{A}:S\rightarrow 2^{S^*}$ satisfies the following conditions: There are constants $C,f,\gamma>0$ such
that
\begin{itemize}
  \item[(A1)] Maximal monotonicity: The map $x\rightarrow \mathcal{A}(x)$ is maximal monotone with non-empty values.
  \item[(A2)] Linear growth:  for all $x\in S$ and $y\in\mathcal{A}(x)$,
$$
\|y\|_{S^*}\leq f+C\|x\|_S,
$$
  \item[(A3)] Weak coercivity in $S$: for all $x\in S$, $y\in\mathcal{A}(x)$ and $n\in\mathbb{N}$,
      $$
      2_S^*\langle y,T_n(x)\rangle_S\geq-\gamma-C\|x\|_S^2.
      $$
\end{itemize}

Suppose that the driving noise $L(t)$, $t\geq0$ satisfies
\begin{itemize}
  \item[(A4)] $L$ is a pure jump L\'evy process taking values in $D(T^{3/2})$, that is, the intensity measure $\nu$ of $L$ satisfies
  $\int_{D(T^{3/2})}\|z\|^2_{D(T^{3/2})}\wedge 1\nu(dz)<\infty$.

  \item[(A5)] $H_0$ is dense in $H$, here $S_\nu$ denotes the support of $\nu$ and
  \begin{equation*}
H_0:=\Big\{\sum_{i=1}^n m_ia_i,\ n,m_1,...,m_n\in\mathbb{N},\ a_i\in S_\nu\Big\}.
\end{equation*}

\end{itemize}

Consider the following deterministic evolution inclusion in $H$:
\begin{eqnarray}\label{eq Ergodicity 1 00}
  &&dX(t)-\cA (X(t)) dt \ni g(t),\ t>0,\\
  &&X(0)=x.\nonumber
\end{eqnarray}
Here $g$ is a $\rm c\grave{a}dl\grave{a}g$ path in $H$.

\begin{defi}
  A solution of \eqref{eq Ergodicity 1 00} is a $\rm c\grave{a}dl\grave{a}g$ function $X\in D([0,\infty);H)$ such that
\begin{eqnarray*}
  X(t)=x+\int_0^t\eta(s)ds+g(t),
\end{eqnarray*}
for all $t\geq0$ as an equation in $S^*$, where $\eta\in L^2_{loc}([0,\infty);S^*)$ such that $\eta(t)\in \mathcal{A}(X(t))$ for a.e. $t\geq 0$.
\end{defi}

\begin{defi}

We say that a function $X\in D([0,\infty);H)$ is a limit solution to \eqref{eq Ergodicity 1 00} with starting point
$x\in H$, if $X(0)=x$ and for each approximation $\{x^\delta\}$ with (eventually) $x^\delta\in S$ and $x^\delta\rightarrow x$ in $H$ the associated solutions
$\{X^\delta\}$ converge to $X$ in $D([0,T];H)$ for any $T>0$.

\end{defi}

\begin{defi}\label{def ergodicity 01}
An $\{\mathbb{F}\}$-adapted stochastic process $X:[0,\infty)\times\Omega\rightarrow H$ is a pathwise (limit) solution to \eqref{eq Ergodicity 1}
with starting point $x\in H$ if for all $\omega\in \Omega$, $X(\omega)$ is a (limit) solution for \eqref{eq Ergodicity 1} with $g(\cdot) = L(\cdot,\omega)$.
\end{defi}

By Theorem 3.2 and the proof of Theorem 2.8  in \cite{GT 2014}, we have the following well-posedness results.
\begin{prop}\label{prop ergo 01}
  Assume that Assumptions (A1)-(A4) hold. We have
  \begin{itemize}
    \item[(1)] For any $X_0\in S$ and $X_0\in\mathcal{F}_0$, there is a unique pathwise solution to \eqref{eq Ergodicity 1} in the sense of
  Definition \ref{def ergodicity 01}.
    \item[(2)] For any $X_0\in X$ and $X_0\in\mathcal{F}_0$, there is a unique pathwise limit solution to \eqref{eq Ergodicity 1} in the sense of
  Definition \ref{def ergodicity 01}. Furthermore, for any $x,y\in H$ and $t\geq0$,
  \begin{eqnarray}\label{eq e prop 01}
 \|X^x(t)-X^y(t)\|_H\leq \|x-y\|_H.
  \end{eqnarray}
 Here $X^x$ and $X^y$ are the limit solutions to  \eqref{eq Ergodicity 1} with initial data $x$ and $y$, respectively.
  \end{itemize}

\end{prop}

Now we state our main results in this subsection.
\begin{thm}\label{thm ergo 1}
  Assume that Assumptions (A1)-(A5) hold. Then there exists at most one invariant measure to the limit solution of the equation
  \eqref{eq Ergodicity 1}.
\end{thm}

\begin{proof}
  The inequality \eqref{eq e prop 01}
implies that the limit solutions $\{X^x,x\in H\}$ satisfy the so-called $e$-property; see \cite{Kapica,KPS 2010}. This together with the irreducibility implies  the uniqueness of the invariant measure (if it exists); see  Theorem 2 in \cite{Kapica}. In the rest of the proof, we will show that $\{X^x,x\in H\}$ is irreducible in $H$.

Since, in general, one can not prove that  the limit solutions $\{X^x,x\in H\}$ satisfy some specific stochastic equations, which plays an important role in the proofs of Theorems \ref{thmmulti} and \ref{thmmulti-additive case}. Hence, we can not apply  Theorems \ref{thmmulti} and \ref{thmmulti-additive case} directly to this case . In the sequel,  with slight modifications to the proof of Theorem \ref{thmmulti}, we will show the irreducibility of $\{X^x,x\in H\}$.

By the Yamada-Watanabe Theorem and Proposition \ref{prop ergo 01}, there exists a measurable map $K:H\times D([0,\infty);H)\rightarrow D([0,\infty);H)$ satisfying
\begin{itemize}
  \item[(B1)] For any  $X_0\in H$ and $X_0\in\mathcal{F}_0$, $K(X_0,L)=X^{X_0}$. Here $X^{X_0}$ is the limit solution of the equation
  \eqref{eq Ergodicity 1} with initial data $X_0$. And for any stopping time $\tau$ and $t\geq0$, $K(K(X_0,L)(\tau),L(\tau+\cdot))(t)=K(X_0,L)(\tau+t)$.
  \item[(B2)] For any $x,y\in H$ and $t\geq 0$, $\|K(x,L)(t)-K(y,L)(t)\|_H\leq \|x-y\|_H$,
  \item[(B3)] Set $G(x,L)(t)=K(x,L)(t)-L(t)-x, t\geq0$, then for any $\omega\in \Omega$, $G(x,L(\omega))\in C([0,\infty);H)$. Hence for any $t\geq0$, $K(x,L)(t)=x+G(x,L)(t)+L(t)$, and $K(x,L)(t)-K(x,L)(t-)=L(t)-L(t-)$.
\end{itemize}
By (B1) and (B2), it is classical that  $\{K(x,L),x\in H\}$ forms a strong Markov process on $H$. Applying (B2) again, using arguments similar to, but much easier than, that in the proof of  Proposition \ref{Prop monotone}, we can see that $\{K(x,L),x\in H\}$ satisfies Assumption \ref{ass2}.
Combining (B3) and Assumption (A5), with only slight modifications to the proof of Theorem \ref{thmmulti}, we obtain the irreducibility of $\{K(x,L),x\in H\}$, that is, the irreducibility of $\{X^x,x\in H\}$.

The proof of Theorem \ref{thm ergo 1} is complete.
\end{proof}

As the application of Theorem \ref{thm ergo 1}, we can obtain the uniqueness of invariant measures of many multi-valued, singular stochastic evolution inclusions. It seems quite difficult to
get these results with other means due to the multi-valued or/and lack of strong dissipativity of the equations.

Here are the examples.

\begin{exmp}[Stochastic singular $\Phi$-Laplace equation in all space dimensions]
\,\,\,\,\,
\begin{itemize}
    \item {\bf Dirichlet boundary conditions in a bounded domain}
\end{itemize}

Let $d\in\NN$, $\Lambda\subset\RR^d$ be a bounded domain with piecewise smooth boundary $\partial\Lambda$ and $-\Delta$ be
the Dirichlet Laplacian on $\Lambda$. We define $|\cdot|$ and $\langle{\cdot},{\cdot}\rangle$ to be the Euclidean norm and the inner-product on $\RR^d$ respectively. Let
  \[
  S:=H^1_0(\Lambda) \subseteq H:=L^2(\Lambda) \subseteq S^*,
  \]
where we endow $S$ with the equivalent norm $\|u\|_{S}:=\|\nabla u\|_{L^2(\Lambda)}$.
Let $\Psi(x)=\widetilde{\Psi}(|x|):\mathbb{R}^d\rightarrow[0,\infty)$ be a radially symmetric function with $\widetilde{\Psi}:\mathbb{R}\rightarrow[0,\infty)$ being even, convex, continuous, non-decreasing and satisfying $\widetilde{\Psi}(0)=0$. Further assume
$$
\Psi(x)\leq C(|x|^2+1),\ \ \forall x\in\mathbb{R}^d
$$
for some constant $C>0$. Let $\Phi:=\partial \Psi:\mathbb{R}^d\rightarrow 2^{\mathbb{R}^d}$.

\begin{itemize}
    \item {\bf Neumann boundary conditions in a bounded domain}
\end{itemize}

Let $d\in\NN$, $\Lambda\subset\RR^d$ be a bounded, open domain with Lipschitz boundary $\partial\Lambda$. We either assume that $\Lambda$ is convex or $\partial\Lambda$
is $C^2$ and convex. Let
  \[
  S:=H^1(\Lambda) \subseteq H:=L^2(\Lambda) \subseteq S^*,
  \]
where  $S$ is normed by $\|u\|^2_{S}:=\|\nabla u\|^2_{L^2(\Lambda)}+\|u\|^2_{L^2(\Lambda)}$. Let $\widetilde{\Psi}(x):\mathbb{R}\rightarrow[0,\infty)$ be an Orlicz function, i.e., $\widetilde{\Psi}$ is even, convex, continuous, non-decreasing and satisfying $\widetilde{\Psi}(0)=0$. Set $\Psi(x):=\widetilde{\Psi}(|x|),\ x\in\mathbb{R}^d$. Assume that for some constant $C>0$
$$
\Psi(x)\leq C(|x|^2+1),\ \ \forall x\in\mathbb{R}^d.
$$
Let $\Phi:=\partial \Psi:\mathbb{R}^d\rightarrow 2^{\mathbb{R}^d}$.

We consider the following stochastic singular $\Phi$-Laplace equation
\begin{equation}\label{SLp-eq}
          dX(t) \in {\rm div} \Phi(\nabla X(t))\,dt + dL(t), t\geq0, \quad
          X_0 =x \in H,
\end{equation}
with Dirichlet or Neumann boundary conditions.

Applying Proposition \ref{prop ergo 01} and Theorem \ref{thm ergo 1},
\begin{prop}
    If the intensity measure $\nu$ of the driving noise $L$ satisfies Assumptions (A4) and (A5).
    Then there exists at most one invariant measure to the limit solution of the equation \eqref{SLp-eq}.

\end{prop}

 The following are several specific examples, and for details, please refer to \cite{GT 2014}.

An explicit example of $\Psi$ in the case of Neumann boundary conditions is given by the singular $p$-Laplacian nonlinearity
$$
\Psi_p(x):=\frac{1}{p}|x|^p,\ x\in\mathbb{R}^d,\ \ p\in[1,2].
$$
Note that when $p=1$  the equation becomes a multi-valued inclusion.

Some examples of $\Phi$ in the case of Dirichlet boundary conditions are as follows.
\begin{enumerate}
  \item Singular $p$-Laplacian:  $\Phi_p(x) := \partial\left(\frac{1}{p}|\cdot|^{p}\right)(x) = |x|^{p-1}\Sgn(x),\ p\in [1,2],$ where
  \[\operatorname{Sgn}(x):=\left\{\begin{aligned}&\dfrac{x}{|x|},&&\;\;\text{if}\;\;x\in\RR^d\setminus\{0\},\\
                 &B(0,1),&&\;\;\text{if}\;\;x=0.
                \end{aligned}\right.\]
  Note that we include the total variation flow, i.e. $p=1$,where the equation becomes a multi-valued inclusion.
  \item Minimal surface flow: \[
  \Phi_{\text{m.s.f.}}(x):=\partial\left(\sqrt{1+|\cdot|^2}\right)(x)=\frac{x}{\sqrt{1+|x|^2}},\quad x\in\RR^d.
  \]
  \item Plastic antiplanar shear deformation:
  \[\Phi_{\text{p.a.s.}}(x):=\partial\left(y\mapsto\begin{cases}\frac{1}{2}|y|^2,&\;\;\text{if}\;\;|y|\le 1,\\ |y|-\frac{1}{2},&\;\;\text{if}\;\;|y|> 1,\end{cases}\right)(x)=\begin{cases}x,&\;\;\text{if}\;\;|x|\le 1,\\ \Sgn(x),&\;\;\text{if}\;\;|x|> 1,\end{cases},\;\; x\in\RR^d.\]
  \item Curve shortening flow:
      \[\Phi_{\arctan}(r) := \partial\left( s\mapsto s\arctan(s)-\frac{1}{2}\log(s^2+1)\right)(r) = \arctan(r),\quad r \in \RR.\]
\end{enumerate}

\end{exmp}

\begin{exmp}[Stochastic generalized fast-diffusion equation] Let $(E,\cB,\mu)$ be a finite measure space and
$(\cE,D(\cE))$ be a symmetric Dirichlet form on $L^2(\mu)$ with associated Dirichlet operator $(\mathcal{L},D(\mathcal{L}))$.
Assume that $\mathcal{L}$ is strictly coercive, self-adjoint, positive-definite and possesses a compact resolvent. Then $D(\cE)$ is a Hilbert space with norm $\|\cdot\|_0:=\cE^{1/2}(\cdot)$ and $D(\cE)\subset L^2(\mu)$ is dense and compact. Let
\[
S := L^2(\mu) \subseteq H:= D(\cE)^\ast\subset S^*,
\]
and $\Psi:\mathbb{R}\rightarrow [0,\infty)$ be an even, convex, continuous function with
$\Psi(0)=0$, subdifferential $\Phi=\partial \Psi:\mathbb{R}\rightarrow 2^{\mathbb{R}}$ and
$$
\Psi(r)\leq C(|r|^2+1_{\{\mu(E)<\infty\}}),\ r\in\mathbb{R}
$$
for some constant $C>0$. We consider the stochastic generalized fast-diffusion equation
\begin{equation}\label{SGFS-eq}
    dX(t) \in \mathcal{L}\Phi(X(t))dt +d L(t), \quad X_0 = x \in H.
\end{equation}

Applying Proposition \ref{prop ergo 01} and Theorem \ref{thm ergo 1}, we obtain the following result.
\begin{prop}
    If the intensity measure $\nu$ of the driving noise $L$ satisfies Assumptions (A4) and (A5).
    Then there exists at most one invariant measure to the limit solution of the equation \eqref{SGFS-eq}.

\end{prop}

Here, an important example of $\mathcal{L}$ is the Laplace operator $\Delta$ on some bounded smooth domain in $\mathbb{R}^d$, $d\in\mathbb{N}$, and $\Phi$ includes two important cases:
  \begin{enumerate}
   \item Fast diffusion equation:
   \[\Phi_p(r) := \partial\left( s\mapsto\frac{1}{p}|s|^{p}\right)(r) = |r|^{p-1}\Sgn(r),\quad p\in [1,2].\] Note that we include the limit case $p=1$ for which $\Phi_1(r) =  \Sgn(r)$, and it is also called sign fast diffusion equation.
   \item Plasma diffusion:
   \[\Phi_{\log}(r) := \partial\Big(s\mapsto(|s|+1)\log(|s|+1)-|s|\Big)(r) = \log(|r|+1)\Sgn(r).\]
  \end{enumerate}
\end{exmp}

\begin{rmk}
In this paper, we focus on the uniqueness of invariant measures. For the existence,   the well-known Krylov-Bogolioubov theorem is a powerful tool. Using the Krylov-Bogolioubov criteria and Theorem  \ref{thm ergo 1} above, we can obtain the ergodicity of the following models:
\begin{itemize}
    \item Stochastic singular $p$-Laplacian:\ $p \in (1,2)$ for $d=1,2$ and $p=1$ for $d=1$;
    \item Stochastic minimal surface flow when $d=1$;
    \item Stochastic plastic antiplanar shear deformation when $d=1$;
    \item Stochastic curve shortening flow;
    \item Stochastic fast diffusion equation: \ $p \in (1,2)$ for $d=1,2$ and $p=1$ for $d=1$.
\end{itemize}
The proof of the existence of invariant measures is similar to that in the Gaussian case.   We refer the reader to \cite{EsRen,Liu-T}.
\end{rmk}


\noindent {\bf Acknowledgement.}
This work is partially supported by the National Key R\&D program of China (No.2022YFA1006001), the National Natural Science Foundation of China (Nos. 12131019, 12371151,
12426655). Jianliang Zhai's research is also supported by the School Start-up Fund (USTC) KY0010000036 and the Fundamental Research Funds for the Central Universities (Nos. WK3470000031, WK0010000081).

\end{document}